\documentclass[a4paper,11pt]{article}

\usepackage[T1]{fontenc}
\usepackage[utf8]{inputenc}
\usepackage[english]{babel}
\usepackage{color}
\usepackage{amsmath,amssymb,amsthm}
\usepackage{ stmaryrd }
\usepackage{latexsym}
\usepackage{graphicx,epsfig}
\usepackage{float,enumerate,array}
\usepackage[left=25mm, right=25mm,
            top=30mm, bottom=30mm]{geometry}
\newtheorem{theo}{Theorem}
\newtheorem{prop}[theo]{Proposition}
\newtheorem{lem}[theo]{Lemma}

\newtheorem{remark}[theo]{Remark}
\newtheorem{coro}[theo]{Corollary}

\newcommand{\bbZ}{\mathbb{Z}}
\newcommand{\set}[1]{\left\{#1\right\}}

\newcommand{\eps}{\varepsilon}
\newcommand{\lf}{\lfloor}
\newcommand{\rf}{\rfloor}

\renewcommand{\P}{\mathbb{P}}

\newcommand{\Z}{\mathbb{Z}}
\newcommand{\R}{\mathbb{R}}

\newcommand{\Lr}{\mathcal{L}}

\usepackage{hyperref,pifont}
\usepackage{todonotes}

\begin{document}

\title{Longest increasing paths with gaps}
\author{%
\textsc{Basdevant A.-L.}\footnote{Laboratoire Modal'X, UPL,  Univ. Paris Nanterre, France. email: anne.laure.basdevant@normalesup.org. Work partially supported by ANR PPPP, ANR Malin and Labex MME-DII.} \and \textsc{Gerin L.}\footnote{CMAP, Ecole Polytechnique, France. email: gerin@cmap.polytechnique.fr.  Work partially supported by ANR PPPP and ANR GRAAL.}}

\maketitle
\begin{abstract}
We consider a variant of the continuous and discrete Ulam-Hammersley problems: we study the maximal length of an increasing path through a Poisson point process (or a Bernoulli point process) with the restriction that there must be minimal gaps between abscissae and ordinates of successive points of the path.

For both cases (continuous and discrete) our approach rely on couplings with well-studied models: respectively the classical Ulam-Hammersley problem and last-passage percolation with geometric weights.
Thanks to these couplings we obtain explicit limiting shapes in both settings.
We also establish that, as in the classical Ulam-Hammersley problem, the fluctuations around the mean are given by the Tracy-Widom distribution.
\end{abstract}
 

\section{Introduction}

Motivated by the Ulam problem (which asks for the asymptotic behavior of the maximal length of an increasing subsequence in a uniform random permutation), 
Hammersley \cite{HammersleyHistorique} studied the problem of the maximal length $L_{(x,t)}$ of an increasing path in a Poisson process with intensity one in $(0,x)\times (0,t)$. He used subadditivity to prove the existence of a constant $\pi/2 \leq c \leq e$ such that $L_{(t,t)}/t \to c$ in probability and conjectured $c=2$.
The first probabilistic proof of $c=2$ was obtained by Aldous-Diaconis \cite{AldousDiaconis}, by exploiting the geometric construction of Hammersley.
(We refer to \cite{Romik} for a nice and modern introduction to the Ulam-Hammersley problem.)
In this article we obtain the limiting behaviour of the maximal length of an increasing path in a Poisson process, if we impose minimal \emph{gaps} between abscissae and ordinates of successive points in the path. 

Our proof uses a coupling with the original Ulam-Hammersley problem, and therefore we make a strong use of Aldous-Diaconis' result. This coupling also allows us to use the celebrated result by Baik-Deift-Johansson regarding the fluctuations of $L_{(x,t)}$ around its mean. We obtain that, with the proper rescaling, the fluctuations of our problem around the mean are also given by the Tracy-Widom distribution.

It turns out that our strategy also applies to the discrete settings: we obtain explicit asymptotic results for the length of the longest increasing path with \emph{gaps} through Bernoulli random points on the square lattice. 
We now state our results.


\subsubsection*{Continuous settings}
Let $\Xi$ be a homogenous Poisson point process in $(0,+\infty)^2$ with intensity $1$. We write $\Xi_{y,s}=0/1$ for the absence/presence of a point of $\Xi$ at $(y,s)$, and we say that $(y,s)\in\Xi$ if $\Xi_{y,s}=1$.
Let ${\bf h}=(h_1,h_2)$ be a pair of non  negative real numbers. 
We introduce the strict order on $(0,+\infty)^2$ defined by
$$
(y,s) \stackrel{{\bf h}}{\prec} (y',s') \text{ if and only if }
\begin{cases}
y+h_1\leq y',\\
s+h_2\leq s'.
\end{cases}
$$
For $x,t>0$, we consider the random variable $L^{\bf h}_{(x,t)}$ given by the length of the longest increasing path in $\Xi \cap[0,x]\times[0,t]$ with horizontal \emph{gaps} $h_1$ and vertical \emph{gaps} $h_2$. Namely,
\begin{multline*}
L^{\bf h}_{(x,t)}=
\max\bigg\{L; \hbox{ there are }(y_1,s_1),\dots,(y_L,s_L)\in \Xi \text{ such that }\\
0<y_1 < \dots < y_L<  x, \quad 0< s_1 < \dots < s_L < t,\quad 
(y_1,s_1) \stackrel{{\bf h}}{\prec} (y_2,s_2)  \stackrel{{\bf h}}{\prec} \dots  \stackrel{{\bf h}}{\prec} (y_L,s_L).
\bigg\}
\end{multline*}
\begin{figure}
\begin{center}
\includegraphics[width=11cm]{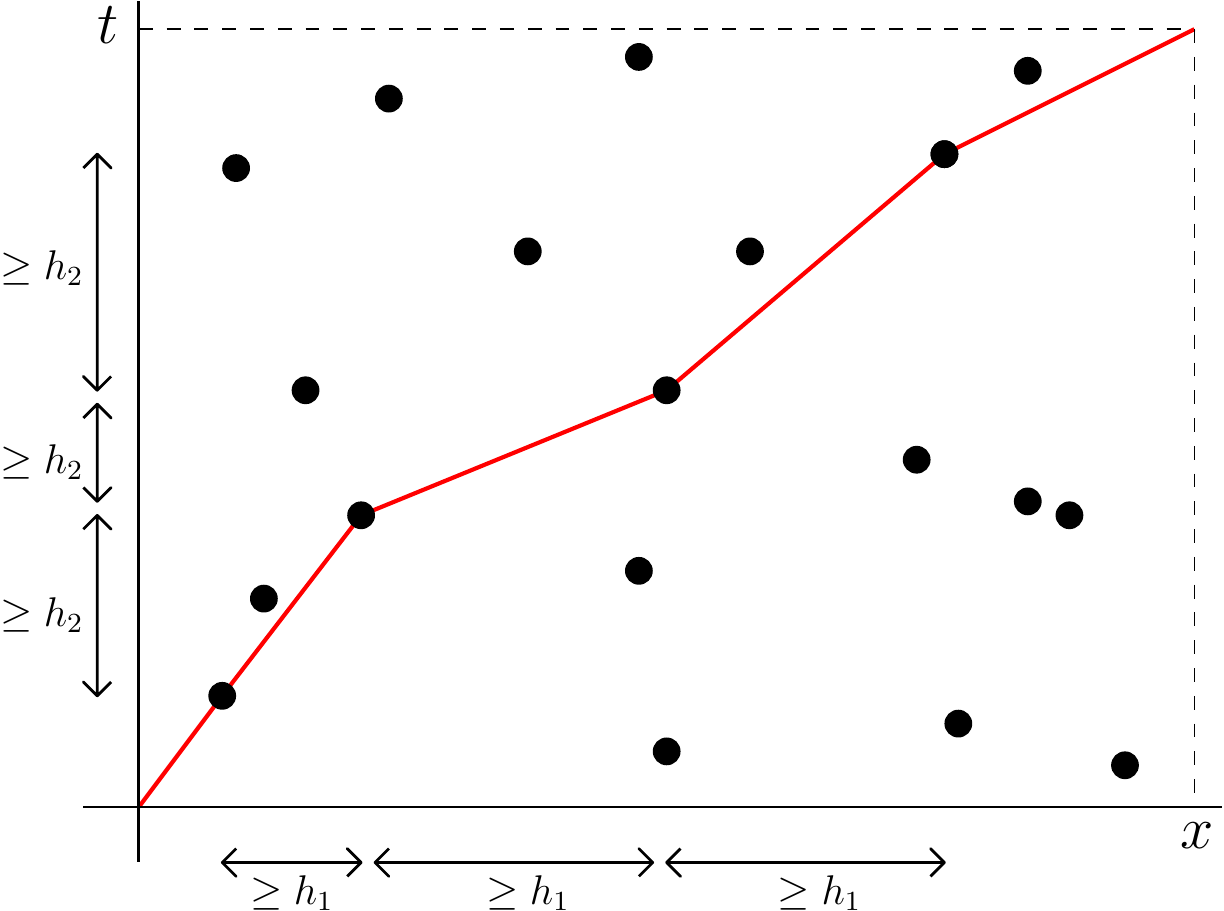}
\end{center}
\caption{A realization of $L^{{\bf h}}_{(x,t)}$  (points of $\Xi$ are represented with $\bullet$). Here we have  $L^{{\bf h}}_{(x,t)}=4$, one of the maximizing paths is drawn in red.}
\label{Fig:DefModele_Continu}
\end{figure}


In the case ${\bf h}={\bf 0}:=(0,0)$, the random variable  $L_{(x,t)}:=L^{{\bf 0}}_{(x,t)}$ is just the length of the longest increasing path.
It turns out that there exists a (random) coupling between $L$ and $L^{{\bf h}}$. As an application, we will show the following identity:

\begin{theo}\label{th:Coupling}
For every $x,t>0$, and every $k \geq 0$,
$$
\mathbb{P}( L^{{\bf h}}_{(x,t)}\le k)= \mathbb{P}( L_{(x-h_1 k,t-h_2 k)}\le k).
$$
\end{theo}
\noindent(In the above equation, we take the convention  $L_{(y,s)}=0$ whenever $y<0$ or $s<0$.)

The asymptotic behavior of $L_{(at,bt)}$ for every $a,b$ was obtained by Aldous-Diaconis \cite{AldousDiaconis}. (Identification of the limit actually dates back to \cite{VershikKerov}, different probabilistic proofs can be found in \cite{Sepp0,Groeneboom}.) 
\begin{theo}[Aldous-Diaconis (\cite{AldousDiaconis}, Th.5 )]
Let $a,b>0$. Then
\begin{equation}\label{eq:LimiteL_continu}
f(a,b):=\lim_{t\to +\infty}\frac{L_{( at,bt)}}{t}= 2\sqrt{ab}.
\end{equation}
The convergence holds a.s. and in $L^1$.
\label{Prop:ConcentrationLn_continu}
\end{theo}

Theorem \ref{th:Coupling} allows us to extend the formula \eqref{eq:LimiteL_continu} to every pair of gaps :
\begin{prop}\label{TheoLLN_continu}
For every $h_1,h_2\geq 0$, we have the following limit:
\begin{equation}\label{eq:LGN_Continu}
f^{{\bf h}}(a,b):= \lim_{t\to \infty} \frac{L^{{\bf h}}_{( at, bt)}}{t} =
\begin{cases}
\displaystyle{\frac{2(ah_2+bh_1)-2\sqrt{(ah_2-bh_1)^2+ab}}{4h_1h_2-1}} &\text{ if }h_1h_2\neq 1/4,\\
\displaystyle{\frac{ab}{h_1b+h_2a}}                                          &\text{ if }h_1h_2= 1/4.
\end{cases}
\end{equation}
The convergence holds a.s. and in $L^1$.
\end{prop}
(We have no probabilistic interpretation of the case $h_1h_2=1/4$, but one can check that the right-hand side of \eqref{eq:LGN_Continu} is continuous at every point of the line  $h_1h_2=1/4$.)

In some cases the above formula for $f^{{\bf h}}(a,b)$ simplifies:
\begin{itemize}
\item If $h_1=h_2=h$, then 
$$
f^{(h,h)}(1,1)=\frac{2}{1+2h}.
$$
\item If $h_2=0$ then 
$$
f^{(h,0)}(1,1)=2\sqrt{h^2+1}-2h.
$$
\end{itemize}

For ${\bf h}=(0,0)$, the fluctuations of $L_{(at,bt)}$ around its mean have been determined by Baik-Deift-Johansson \cite{BDK}.
\begin{theo}[Baik-Deift-Johansson \cite{BDK}]\label{Prop:Flu_Ln_continu}
For every $a,b>0$ and $x\in \R$,we have 
\begin{equation*}
\lim_{t\to \infty}\P\left(\frac{L_{(at,bt)} -2\sqrt{ab}t}{(\sqrt{ab}t)^{1/3}}\le x\right)
=F_{TW}(x),
\end{equation*}
where $F_{TW}$ is the distribution function of the Tracy-Widom distribution.
\end{theo}
In fact the main result of \cite{BDK} is stated for the longest increasing subsequence in a uniform permutation. Theorem \ref{Prop:Flu_Ln_continu} follows by elementary poissonization arguments. This theorem can also been extended for every pair of gaps :
\begin{prop}\label{Theoflu_continu}
For every $h_1,h_2,a,b\geq 0$ and $x\in \R$, we have 
\begin{equation}\label{eq:Flu_Continu}
\lim_{t\to \infty}\P\left(\frac{L^{\bf h}_{(at,bt)} -f^{{\bf h}}(a,b)t}{\sigma^{\bf h}(a,b)t^{1/3}}\le x\right)
= F_{TW}(x),
\end{equation}
where $F_{TW}$ is the distribution function of the Tracy-Widom distribution and
$$\sigma^{\bf h}(a,b)= \frac{f^{{\bf h}}(a,b)^{4/3}}{2^{1/3}}\frac{1}{2(bh_1+ah_2)+f^{{\bf h}}(a,b)(1-4h_1h_2)}.$$
\end{prop}
\noindent(In some cases the expression for $\sigma^{\bf h}(a,b)$ simplifies, for instance $\sigma^{(h,h)}(1,1)=(1+2h)^{-4/3}$.)


Thanks to the scale-invariance property of the Poisson process we also easily obtain asymptotic results in the case where gaps and intensity of the Poisson process depend on $t$ (see Section \ref{Sec:EcartQuiBouge}).

\subsubsection*{Discrete settings}

The  same strategy allows us to obtain analogous results in the discrete settings. Let $\Xi=\left(\Xi_{i,j}\right)_{i,j\in \bbZ_{> 0}}$ be  i.i.d. Bernoulli random variables  with mean $p$. We also consider  $\Xi$ as a random set of integer points of the quarter-plane by saying that $(i,j)$ is present in $\Xi$ if $\Xi_{i,j}=1$.

Let ${\bf h}=(h_1,h_2)$ be a pair of non-negative integers, we assume ${\bf h}\neq {\bf 0}=(0,0)$.
We introduce the strict order on $(\bbZ_{>0})^2$ defined by
$$
(i,j) \stackrel{{\bf h}}{\prec} (i',j') \text{ if and only if }
\begin{cases}
i+h_1\leq i',\\
j+h_2\leq j'.
\end{cases}
$$

We consider the random variable given by the length of the longest non-decreasing path from $(1,1)$ to $(m,n)$ in $\Xi$ with horizontal gaps $h_1$ and vertical gaps $h_2$. Namely,
\begin{multline*}
\Lr^{\bf h}_{(m,n)}=
\max\bigg\{L; \hbox{ there are }(i_1,j_1),\dots,(i_L,j_L)\in \Xi \text{ such that }\\
1\leq i_1 \leq \dots \leq i_L\leq m, \quad 1\leq j_1 \leq \dots \leq j_L\leq m,\quad 
(i_1,j_1) \stackrel{{\bf h}}{\prec} (i_2,j_2)  \stackrel{{\bf h}}{\prec} \dots  \stackrel{{\bf h}}{\prec} (i_L,j_L)
\bigg\}.
\end{multline*}
This problem is close to what is sometimes called \emph{slope-constrained longest increasing subsequence} (SCLIS) in the literature of algorithms \cite{SCLIS}. Two particular cases have received particular attention:
\begin{itemize}
\item If ${\bf h}=(1,1)$,  $\Lr^{(1,1)}_{(m,n)}$ is the length of the longest increasing path in $\Xi$.
\item If ${\bf h}=(1,0)$,  $\Lr^{(1,0)}_{(m,n)}$ is the length of the longest non-decreasing path.
\end{itemize}
Both problems have been first studied by Sepp\"al\"ainen (resp. in \cite{Sepp} and \cite{Sepp2}).

Similarly to our approach for the continuous settings, we will use a coupling between $\Lr^{\bf h}$ and a well-studied model: last passage percolation with geometric weights (which in turn is in correspondence with synchronous TASEP). 

Let us recall formally this latter model. Let $\Xi'=\left(\Xi'_{i,j}\right)_{i,j\in \bbZ_{> 0}}$ be  i.i.d. geometric random variables  with law 
$$\P(\Xi'_{i,j}=k)=p^k(1-p) \mbox{ for } k\ge 0$$ 
and let 
\begin{equation*}
T_{(m,n)}= 
\max \{ \sum_{(i,j)\in P} \Xi'_{i,j} \; ; \; P\in \mathcal{P}_{m,n}\},
\end{equation*}
where $\mathcal{P}_{m,n}$ denotes the set of paths from $(1,1)$ to $(m,n)$ taking only North and East steps. 

The discrete analogous of Theorem \ref{th:Coupling} is the following:

\begin{theo}\label{th:DiscreteCoupling}
Let ${\bf h}=(h_1,h_2)\neq (0,0)$. For every $m,n\geq 0$, and every $k \geq 0$,
$$
\mathbb{P}( \Lr^{\bf h}_{(m,n)}\le k)= \mathbb{P}( T_{(m-h_1 k,n-h_2 k)}\le k).
$$
\end{theo}

\begin{remark}
As we will see in the proof of Theorem \ref{th:DiscreteCoupling}, many relevant quantities regarding $T$ are obtained by taking formally ${\bf h}= (0,0)$ in the formulas for $\Lr^{\bf h}$. As suggested by Theorem \ref{th:DiscreteCoupling}, the key difference is that the Bernoulli point process must replaced by a point process with geometric weights.
 
Therefore we cannnot extend our methods and results to the study of $\Lr^{(0,0)}$. This last model can be seen as (directed) site percolation on the quarter-plane, for which the critical threshold remains unknown.
\end{remark}

Theorem \ref{th:DiscreteCoupling}  is related to previous results in literature.
In the case ${\bf h}=(1,1)$, a similar coupling was implicit in \cite{Dhar98} (see also \cite{GeorgiouOrtmann,MajumdarNechaevEARMSA}).
Still in the case ${\bf h}=(1,1)$ another coupling between with \emph{asynchronous} TASEP (also called directed TASEP) was also given in \cite[Sec.3]{Georgiou},\cite{Prieehev}.

\begin{figure}
\begin{center}
\begin{tabular}{c c}
\includegraphics[width=8cm]{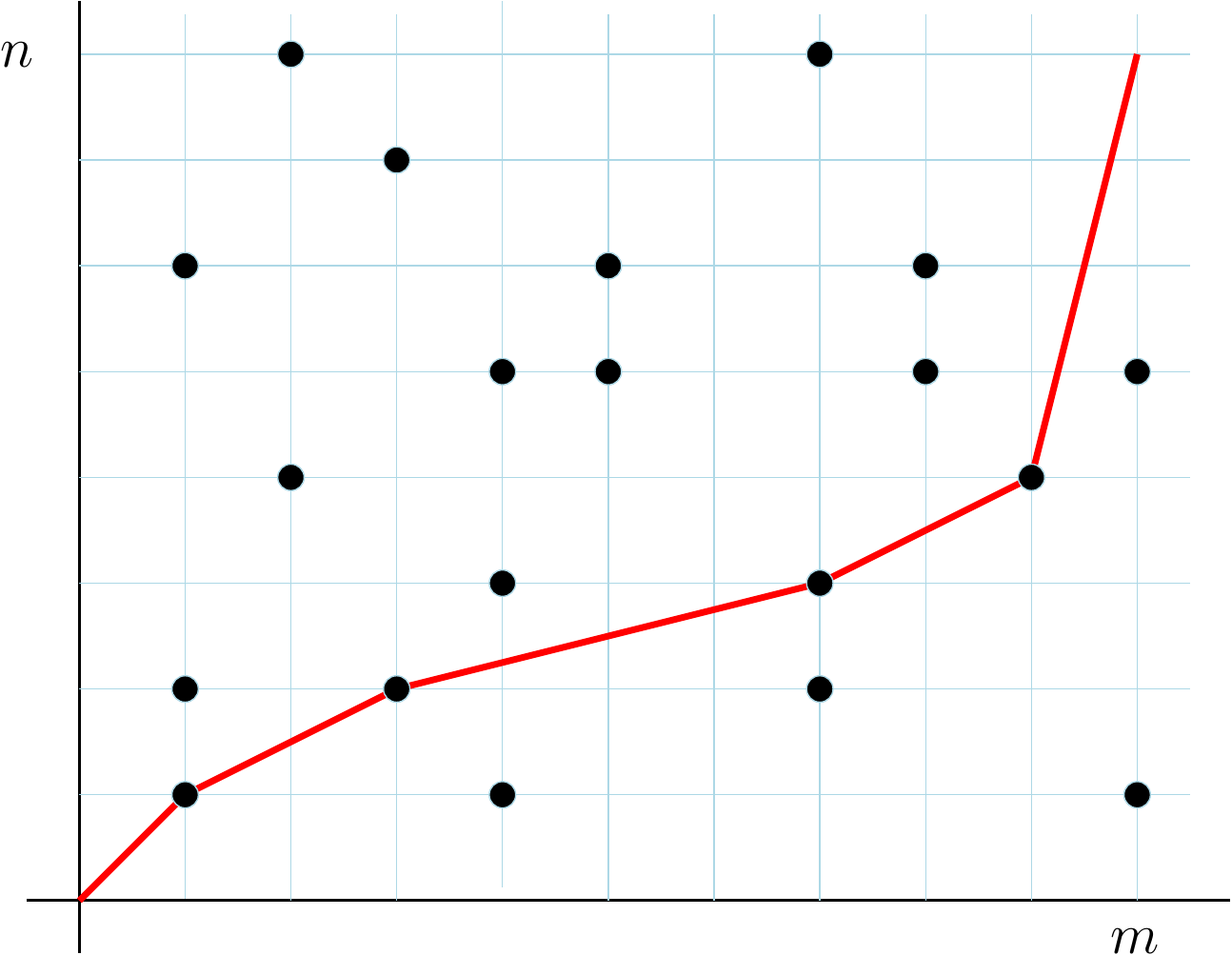} &
\includegraphics[width=8cm]{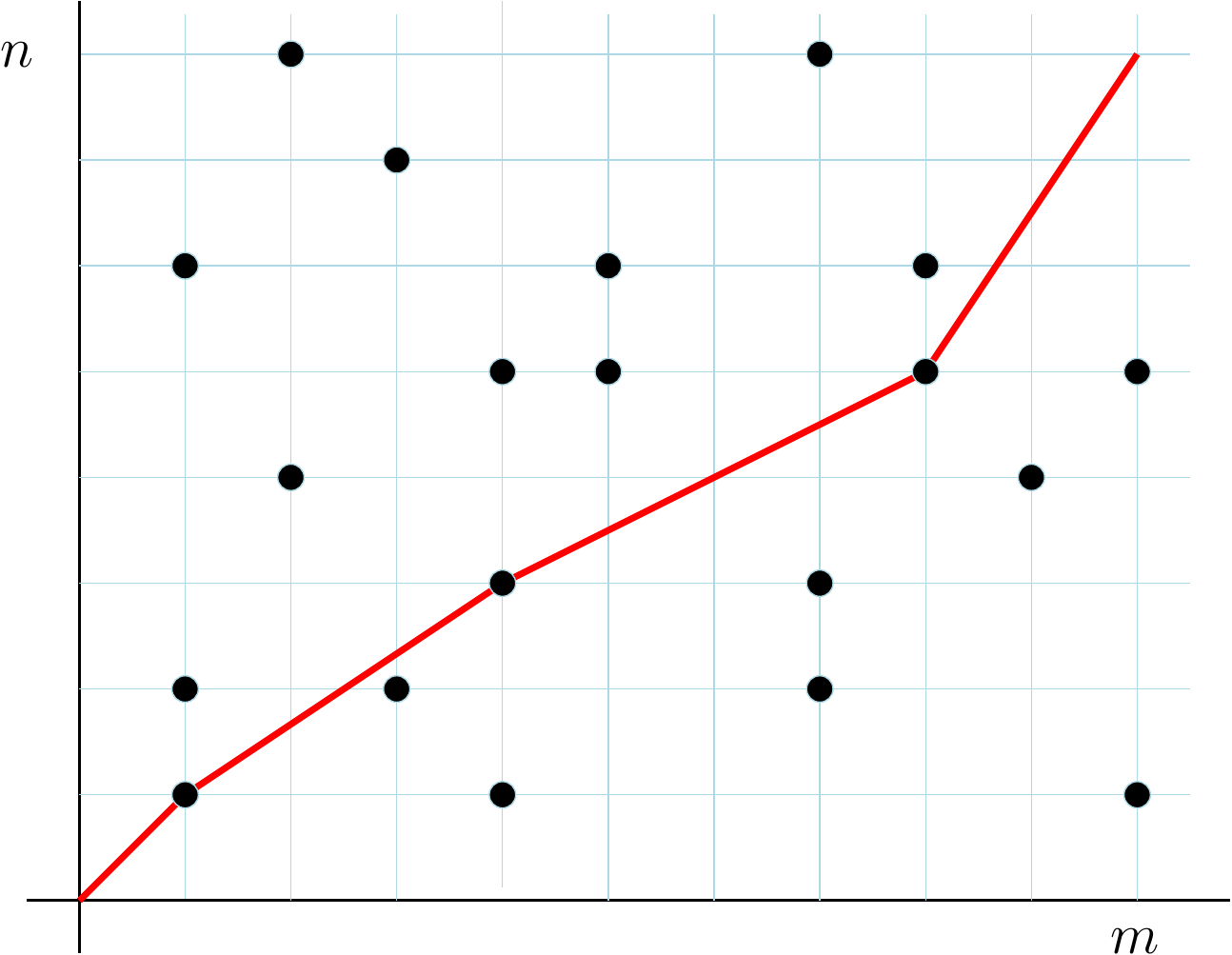}
\end{tabular}
\end{center}
\caption{Realizations of $\Lr^{(2,1)}_{(m,n)}$ (left) and $\Lr^{(3,2)}_{(m,n)}$ (right) for $(m,n)=(10,8)$ and the same sampling of $\Xi$ (points of $\Xi$ are represented with $\bullet$). Here we have $\Lr^{(2,1)}_{(m,n)}=4$, $\Lr^{(3,2)}_{(m,n)}=3$. In both pictures one of the maximizing paths is drawn in red.}
\label{Fig:SansChemin}
\end{figure}

The explicit formula for the limiting shape in last-passage percolation with geometric weights is originally due to Jockusch-Propp-Shor (in the context of synchronous TASEP).
\begin{theo}[Jockusch-Propp-Shor (\cite{ArcticCircle},Th.2), see also (\cite{Sepp3},Th.2.2)]\label{Th:JockuschProppShor}
For every $a,b>0$, $p\in(0,1)$
\begin{equation}\label{eq:gTASEP}
g(a,b):=\lim_{n\to +\infty} \frac{1}{n} T_{(\lf an\rf,\lf bn\rf)} = \frac{\sqrt{p}\left(2\sqrt{ab}+(a+b)\sqrt{p}\right)}{1-p}. 
\end{equation}
\end{theo}

Theorem \ref{th:DiscreteCoupling} then allows us to  deduce the limiting shape for $\Lr^{{\bf h}}$. The limiting constant is less explicit than in the continuous settings.
\begin{prop}\label{TheoLLN}
Let $h_1,h_2$ be two non-negative integers such that ${\mathbf h}=(h_1,h_2)\neq (0,0)$. For every $a,b>0$, there exists a constant
$$
g^{{\bf h}}(a,b):=\lim_{n\to +\infty} \frac{1}{n}\Lr^{{\bf h}}_{(\lf an\rf,\lf bn\rf)},
$$
where the convergence holds a.s. and in $L^1$. 
Moreover, we have
\begin{itemize}
\item If $\frac{h_1}{(h_2-1)+1/p}<\frac{a}{b}<\frac{h_1-1+1/p}{h_2}$, $g^{{\bf h}}(a,b)$ is the unique solution of equation
\begin{equation}\label{Systeme_xy}
g^{{\bf h}}(a,b)= g\left(a-h_1g^{{\bf h}}(a,b),b-h_2g^{{\bf h}}(a,b) \right),
\end{equation}
where $g$ is defined by \eqref{eq:gTASEP} ;
\item If $\frac{a}{b}\le\frac{h_1}{(h_2-1)+1/p}$, $g^{{\bf h}}(a,b)=\frac{a}{h_1}$, 
\item If $\frac{a}{b}\ge\frac{h_1-1+1/p}{h_2}$, $g^{{\bf h}}(a,b)=\frac{b}{h_2}$.
\end{itemize}
\end{prop}
The second and third cases correspond to a \emph{flat edge} in the limiting shape. This differs from the continuous case.

We explicit here the solution of \eqref{Systeme_xy} in some cases: 
\begin{itemize}
\item {\bf Increasing paths.} For ${\bf h}=(1,1)$ the above formula reduces to
$$
g^{(1,1)}(a,b)=
\begin{cases}
\displaystyle{\frac{\sqrt{p}\left(2\sqrt{ab}-(a+b)\sqrt{p}\right)}{1-p}} &\text{ if }p< \min\set{a/b,b/a},\\
\min\set{a,b} &\text{ otherwise.}\\
\end{cases}.
$$
Thus we recover the asymptotic behavior of $\Lr^{(1,1)}_{(\lf an\rf,\lf bn\rf)}$ which was obtained by Sepp\"al\"ainen in \cite{Sepp} using hydrodynamic limits of a given particle system (see also \cite[Sec.III]{Majumdar}, \cite[Th.1.1.]{NousAlea} ,\cite[Th.2.2]{CiechGeorgiou} for different proofs).
\item {\bf Non-decreasing paths.} If ${\bf h}=(1,0)$ the above formula reduces to
$$
g^{(1,0)}(a,b)= 
\begin{cases}
2\sqrt{abp(1-p)}+(a-b)p &\text{ if }p< a/(a+b),\\
a &\text{ otherwise.}\\
\end{cases}
$$
This was also first proved by Sepp\"al\"ainen in a second article (\cite{Sepp2}, Th.1) (see again \cite{NousAlea} for a different proof). In the more general case ${\bf h}=(h,0)$ we obtain with \eqref{Systeme_xy} the following expression (we only write the formula for $a=b=1$):
$$
g^{(h,0)}(1,1)=
\begin{cases}
 \displaystyle{\frac{2(1+h)p(1-p)+2\sqrt{p(1-p+h^2p)(1-p)}}{\left(h\sqrt{p}+\sqrt{(1-p+h^2p)(1-p)}\right)^2}} & \text{ if }p<1/(h+1),\\
 1/h  & \text{ otherwise.}
\end{cases}
$$
\item {\bf Symmetric case.} If $(a,b)=(1,1)$ and $h_1=h_2=h$ we can easily solve \eqref{Systeme_xy} and we get
$$
g^{(h,h)}(1,1)=   \frac{2\sqrt{p}}{1+(2h-1)\sqrt{p}}.
$$
\end{itemize}

\section{Proofs in the continuous settings}
We fix a pair ${\bf h}=(h_1,h_2)$ of non-negative real numbers all along this section.
\subsection{Preliminary results}
We first justify that $\frac{1}{t}L^{{\bf h}}_{(at,bt)}$ converges almost surely and in $L^1$.
Let us stress that for any $x,x',t,t'\ge 0$, we have the stochastic domination
$$L^{{\bf h}}_{(x+x',t+t')}\succcurlyeq L^{{\bf h}}_{(x,t)} +L'^{{\bf h}}_{(x',t')}-1,$$
where $L'^{{\bf h}}_{(x',t')}$ has the same  distribution as $L^{{\bf h}}_{(x',t')}$ but is independent of $L^{{\bf h}}_{(x,t)}$. 

Indeed, if $(i_1,j_1),\dots,(i_L,j_L)$ is a longest increasing path in $\Xi$ with gaps $\bf h$ in $(0,x)\times(0,t)$ and $(i'_1,j'_1),\dots,(i'_{L'},j'_{L'})$   a longest increasing path in $(x,x+x')\times(t,t+t')$ 
, then 
$$
(i_1,j_1),\dots,(i_L,j_L),(i'_2,j'_2),\dots,(i'_{L'},j'_{L'})
$$
is an  increasing path  with  gaps $h$ in $(0,x+x')\times(0,t+t')$.
Thus, the family of random variables $\set{L^{{\bf h}}_{(x,t)}-1}_{x>0,t>0}$ is superadditive. Hence, Kingman's subadditive theory (see for example \cite[Th.A2-A3]{Romik}) implies the existence of a constant
$$
f^{{\bf h}}(a,b) := \lim_{t \to \infty} \frac{L^{{\bf h}}_{(at,bt)}}{t},$$
where the limit is a.s. and in $L^1$.

\subsection{Hammersley's lines and dilatation}

A very useful way to handle the random variables $L^{\bf h}_{(x,t)}$ is the geometric interpretation of Hammersley's lines. In the classical case ${\bf h}={\bf 0}$ this construction was first implicitly introduced by Hammersley \cite{HammersleyHistorique}, a more explicit construction was given by Aldous-Diaconis in \cite{AldousDiaconis} (continuous settings) and by Sepp\"al\"ainen \cite{Sepp} (discrete settings).

\begin{figure}[h!]
\begin{center}
\includegraphics[width=10cm]{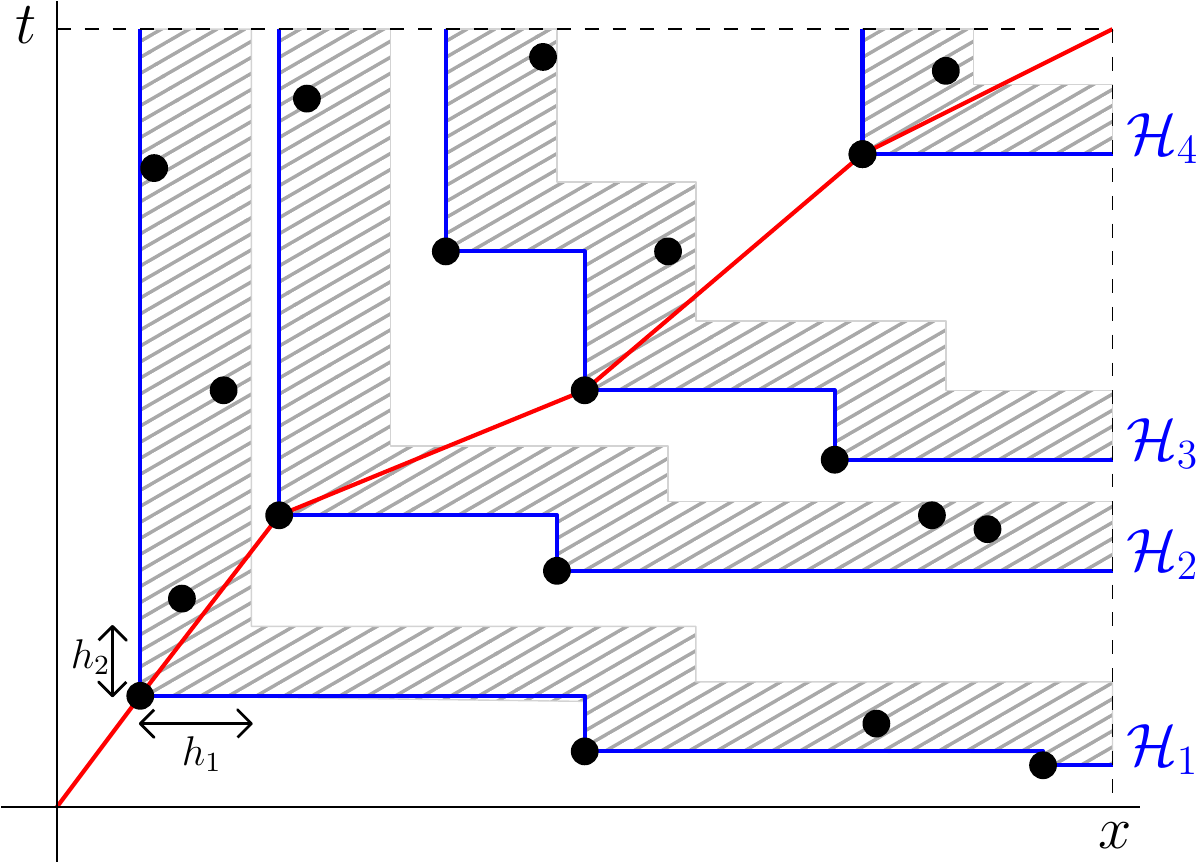}
\end{center}
\caption{An example of Hammersley lines for the same realization of $\Xi$ as that of Fig.\ref{Fig:DefModele_Continu}. The four Hammersley lines are drawn in blue. }
\label{Fig:DefModele_Continu_AvecLignes}
\end{figure}

We now define Hammersley lines formally. These are a sequence $\mathcal{H}^{\bf h}_1,\mathcal{H}^{\bf h}_2,\dots$ of broken lines in $(0,+\infty)^2$ defined inductively as follows (an example is provided in Fig.\ref{Fig:DefModele_Continu_AvecLignes}).

The broken line $\mathcal{H}^{\bf h}_1$ is the shortest path made of vertical and horizontal straight lines whose minimal points for $\stackrel{{\bf 0}}{\prec}$ are exactly  the minimal points of $\Xi$ for  $\stackrel{{\bf 0}}{\prec}$.

The  line $\mathcal{H}^{\bf h}_2$ is defined as follows: we remove the points of $\mathcal{H}_1^{\bf h}+[0,h_1]\times[0,h_2]$ (hatched in gray in Fig.\ref{Fig:DefModele_Continu_AvecLignes})  and reiterate the procedure: $\mathcal{H}^{\bf h}_2$ is the shortest path made of vertical and horizontal straight lines whose minimal points for $\stackrel{{\bf 0}}{\prec}$ are exactly  the minimal points of $\Xi\setminus \left(\mathcal{H}^{\bf h}_1+[0,h_1]\times[0,h_2]\right)$  for  $\stackrel{{\bf 0}}{\prec}$.
Inductively we define $\mathcal{H}^{\bf h}_3,\mathcal{H}^{\bf h}_4,\dots$ in the same way.

\begin{lem}
For each $(x,t)\in (0,+\infty)^2$, there are exactly  $L^{\bf h}_{(x,t)}$ distinct Hammersley lines which intersect $(0,x)\times (0,t)$.
\label{lem:NbLignes}
\end{lem}
\begin{proof}[Proof of Lemma \ref{lem:NbLignes}]
Let denote by  $\mathcal{N}_{(x,t)}$ the number of distinct Hammersley lines which intersect $(0,x)\times (0,t)$.

\noindent {\bf Proof of }$L^{\bf h}_{(x,t)}\leq \mathcal{N}_{(x,t)}$. Let $P=(y_1,s_1)\prec \dots \prec(y_{L_{(x,t)}},s_{L_{(x,t)}})$ be a maximizing path in $\Xi$ for $L_{(x,t)}$. For every $\ell \leq \mathcal{N}_{(x,t)}$, there is at most one point of $P$ in the area $\mathcal{H}_\ell^{\bf h}+[0,h_1]\times[0,h_2]$. Therefore $L^{\bf h}_{(x,t)}\leq \mathcal{N}_{(x,t)}$.

\noindent {\bf Proof of }$L^{\bf h}_{(x,t)}\geq \mathcal{N}_{(x,t)}$. Let  $\mathcal{H}^{\bf h}_1,\mathcal{H}^{\bf h}_2,\dots,\mathcal{H}^{\bf h}_\ell$ be given, we will construct an admissible path with $\ell$ points of $\Xi$ (from top-right to bottom-left).
 We first take any point $(y_\ell,s_\ell)$ of $\mathcal{H}^{\bf h}_\ell$. Let $\hat{\mathcal{H}}=\mathcal{H}^{\bf h}_{\ell-1}+(h_1,h_2)$ be the translation of $\mathcal{H}^{\bf h}_{\ell-1}$ by the gaps. 
By construction of $\mathcal{H}^{\bf h}_{\ell}$ the broken line $\hat{\mathcal{H}}$ intersects  $(0,y_\ell)\times(0,s_\ell)$.
Since $\hat{\mathcal{H}}$ takes only directions North/West, necessarily there is a point $(y_{\ell -1}+h_1,s_{\ell-1}+h_2)\in \hat{\mathcal{H}}\cap (0,y_\ell)\times(0,s_\ell)$, with $(y_{\ell -1},s_{\ell-1})\in\Xi$ .

\begin{center}
\includegraphics[width=6cm]{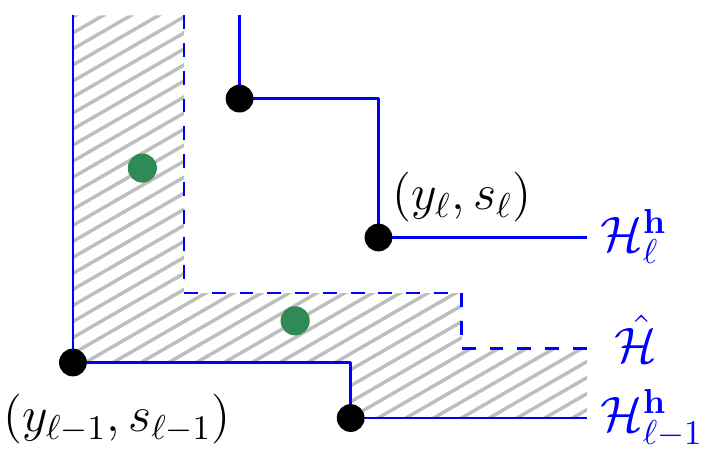}
\end{center}

Therefore, $(y_{\ell -1},s_{\ell-1})\stackrel{{\bf h}}{\prec} (y_{\ell},s_{\ell})$. By induction we construct an admissible path of $\ell$ points in $\Xi$.

\end{proof}



The following Proposition was used in \cite{AldousDiaconis} in the case ${\bf h}={\bf 0}$, it extends plainly to the general case. 
\begin{prop}[spatial Markov property for Hammersley's lines]
Conditional on the $\ell$-th Hammersley line $\mathcal{H}^{\bf h}_\ell$, 
$$
\set{\Xi_{y,s},\ (y,s)\stackrel{{\bf 0}}{\succ} \mathcal{H}^{\bf h}_\ell}
\text{ is independent of }
\set{\Xi_{y,s},\ (y,s)\stackrel{{\bf 0}}{\prec} \mathcal{H}^{\bf h}_\ell},
$$
and distributed as a homogeneous Poisson process with intensity one.
Here $(y,s)\stackrel{{\bf 0}}{\succ} \mathcal{H}^{\bf h}_\ell$ (resp. $\stackrel{{\bf 0}}{\prec}$) means that $(y,s)\stackrel{{\bf 0}}{\succ} (y',s')$ (resp.~$\stackrel{{\bf 0}}{\prec}$) for at least one point $(y',s')$ in $\mathcal{H}^{\bf h}_\ell$.

In particular, conditional on  $\mathcal{H}^{\bf h}_\ell$, the line $\mathcal{H}^{\bf h}_{\ell+1}$ is independent of $\mathcal{H}^{\bf h}_1,\mathcal{H}^{\bf h}_2,\dots,\mathcal{H}^{\bf h}_{\ell-1}$.
\label{lem:Markovspatial}
\end{prop}
\begin{proof}
Let $\mathcal{H}^{\bf h}_1,\mathcal{H}^{\bf h}_2,\dots,\mathcal{H}^{\bf h}_{\ell-1}$ be given.
By construction of Hammersley lines, the fact that $(y,s)$ belongs to $\mathcal{H}^{\bf h}_{\ell}$ or not only depends on $\Xi$ in the rectangle $[0,y]\times [0,s]$.
\end{proof}

We want to make a coupling between random variables $L_{(x,t)}$ and $L^{\bf h}_{(x',t')}$ for some $x'\geq x$, $t'\geq t$. We fix a realization of $\Xi$ in the quarter-plane, and denote by $\set{L_{(x,t)}(\Xi),\ (x,t)\in (0,+\infty)^2}$ the lengths of the longest paths corresponding to this realization.

We introduce the (random) function
$$
\begin{array}{r c c c}
\phi^{\bf h}: & (0,+\infty)^2 &   \to   &    (0,+\infty)^2 \\
      & (y,s)         & \mapsto & (y+h_1 L_{(y,s)^-},s+h_2L_{(y,s)^-}),
\end{array}
$$
where $L_{(y,s)^-} = \lim_{\eps \to 0} L_{(y-\eps,s-\eps)}$.

An example is drawn in Fig.\ref{Fig:FonctionPhi_PourLemmeInprouvable}.
By construction, the image by $\phi^{\bf h}$ of every Hammersley line $\mathcal{H}^{{\bf 0}}_\ell$ is a translation of $\mathcal{H}^{{\bf 0}}_\ell$ (and the area between two consecutive Hammersley lines $\mathcal{H}^{{\bf 0}}_\ell,\mathcal{H}^{{\bf 0}}_{\ell +1}$ is also translated by $\phi^{\bf h}$). 

\begin{figure}
\begin{center}
\includegraphics[width=14cm]{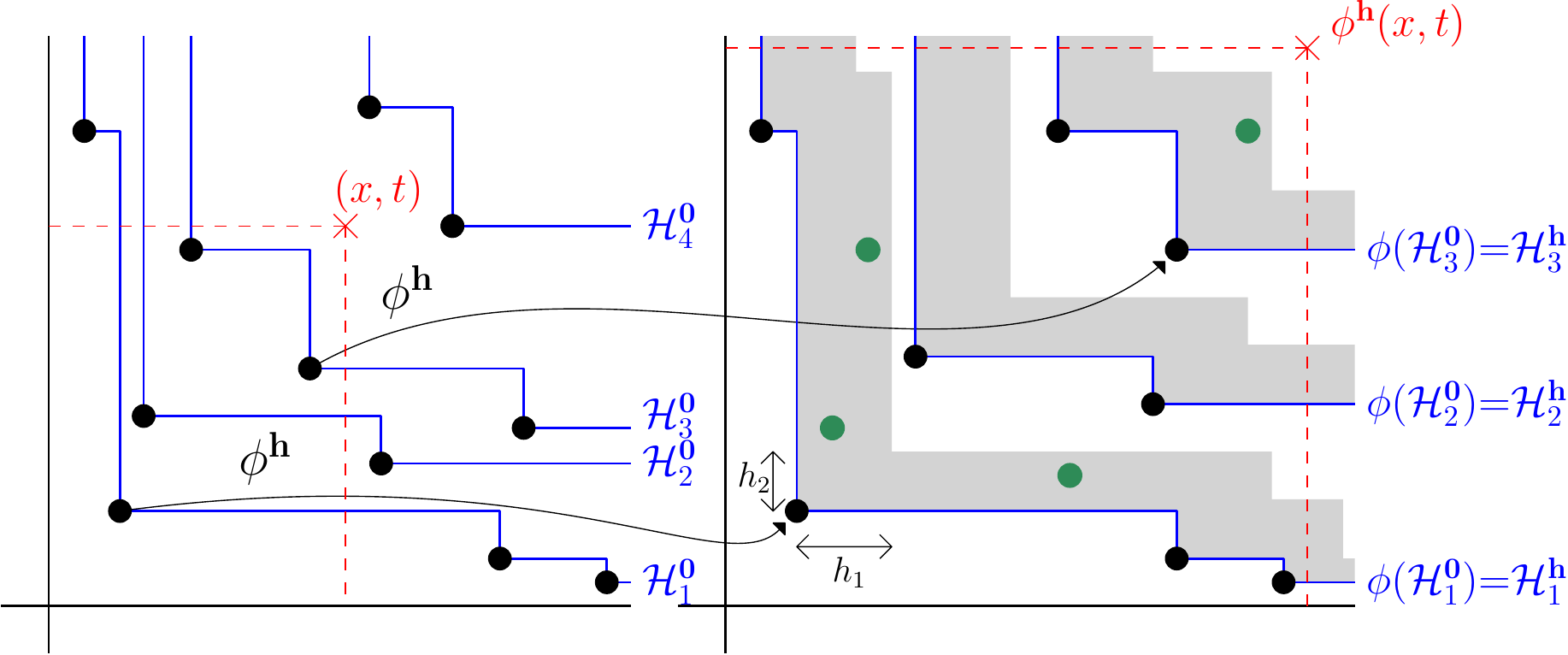}
\end{center}
\caption{An example of the function $\phi^{\bf h}$. Left: A sample of $\Xi$. Right: The same realization after the dilatation $\phi^{\bf h}$. 
The gray areas correspond to regions which are not in the image of $\phi^{\bf h}$. New points of $\tilde{\Xi}\setminus \phi^{\bf h}(\Xi)$ are drawn in green.
We have that $L_{(x,t)}(\Xi)= 3=L^{{\bf h}}_{\phi^{\bf h}(x,t)}(\tilde{\Xi})$, as stated in Eq.\eqref{eq:Coupling}.
}
\label{Fig:FonctionPhi_PourLemmeInprouvable}
\end{figure}

The main idea is that on the right picture of Figure \ref{Fig:FonctionPhi_PourLemmeInprouvable} we re-sample new points in the regions which do not belong to the image of $\phi^{\bf h}$, according to an independent Poisson process.

\begin{lem}[Dilatation]
Let $\tilde{\Xi}$ be the field of random points defined by
$$
\begin{cases}
\tilde{\Xi}_{y',s'}&= \Xi_{y,s} \text{ if }(y',s')\in\mathrm{Image}(\phi^{\bf h}) \text{ and }\phi^{\bf h}(y,s)=(y',s'),\\
\tilde{\Xi}_{y',s'}&= \mathbf{Y}_{y',s'} \text{ if }(y',s')\notin\mathrm{Image}(\phi^{\bf h}),
\end{cases}
$$
where $\mathbf{Y}$ is a homogeneous Poisson process with intensity one, independent of $\Xi$.
Then $\tilde{\Xi}$ is also a homogeneous Poisson process with intensity one.
\label{lem:NouveauXi}
\end{lem}

\begin{proof}[Proof of Lemma \ref{lem:NouveauXi}]

The point process $\tilde{\Xi}$ can be discovered by the following Markovian exploration of $(0,+\infty)^2$.

From bottom-left to top-right, both point processes $\Xi$ and $\tilde{\Xi}$ coincide up to $\mathcal{H}^{{\bf 0}}_1$. 
Then, conditional to $\mathcal{H}^{{\bf 0}}_1$, the points in the area $\phi^{\bf h}(\mathcal{H}^{{\bf 0}}_1)+(0,h_1)\times(0,h_2)$ (this is the first gray region in Fig.\ref{Fig:FonctionPhi_PourLemmeInprouvable}) are also distributed as an independent homogeneous Poisson process.
This shows that, conditional to $\mathcal{H}^{{\bf 0}}_1$, $\tilde{\Xi}$ is also a Poisson process up to $\phi^{\bf h}(\mathcal{H}^{{\bf 0}}_1)+(0,h_1)\times(0,h_2)$. But now, thanks to the Markovian property of Hammersley lines applied to $\mathcal{H}^{{\bf 0}}_1$
we can reiterate the argument to show that conditional to $\mathcal{H}^{{\bf 0}}_2$, $\tilde{\Xi}$ is a Poisson process up to $\phi^{\bf h}(\mathcal{H}^{{\bf 0}}_2)+(0,h_1)\times(0,h_2)$, and so on.

\end{proof}


\begin{proof}[Proof of Theorem \ref{th:Coupling}]
We will first prove that almost surely, for every $x,t$ we have
\begin{equation}\label{eq:Coupling}
L_{(x,t)}(\Xi)= L^{{\bf h}}_{(x',t')}(\tilde{\Xi}),
\end{equation}
where $(x',t')=\phi^{\bf h}(x,t)$ \emph{i.e.}
$$
(x',t')=(x+h_1L_{(x,t)^-}(\Xi),t+h_2L_{(x,t)^-}(\Xi)).
$$

The quarter-plane $(0,+\infty)^2$ is divided in two types of regions:
\begin{enumerate}
\item The region $\mathcal{W}$ defined by the interior of $\phi^{\bf h}((0,+\infty)^2)$ (represented in white in Fig.\ref{Fig:FonctionPhi_PourLemmeInprouvable}). There are no points of $\tilde{\Xi}$ in $\mathcal{W}$.
\item The regions $\mathcal{G}_\ell=\phi^{\bf h}(\mathcal{H}^{\bf 0}_{\ell})+[0,h_1]\times[0,h_2]$, for $\ell\geq 1$ (represented in gray in Fig.\ref{Fig:FonctionPhi_PourLemmeInprouvable}). 
\end{enumerate}

Let $P=(x_1,s_1)\stackrel{{\bf 0}}{\prec} \dots \stackrel{{\bf 0}}{\prec}(y_{L_{(x,t)}},s_{L_{(x,t)}})$ be a maximizing path in $\Xi$.
Because of the dilatation, the points of $\phi^{\bf h}(P)$ have horizontal gaps $h_1$ and vertical gaps $h_2$. Therefore the path $\phi^{\bf h}(P)$ satisfies the gaps constraints and
$$
L_{(x,t)}(\Xi)\leq L^{{\bf h}}_{(x',t')}(\tilde{\Xi}).
$$
For the reverse inequality, we observe that because of the gaps constraint, an admissible path for the order $\stackrel{{\bf h}}{\prec}$ takes at most one point in each $\mathcal{G}_\ell$. Since there are $L_{(x,t)}(\Xi)$ gray regions which intersect $(0,x')\times(0,t')$, this proves that $L^{{\bf h}}_{(x',t')}(\tilde{\Xi})\leq L_{(x,t)}(\Xi)$.
Finally we have proved \eqref{eq:Coupling}.

We now conclude the proof of the theorem. Let $\gamma \geq 0$ be such that 
$$
\gamma =\sup \set{y \geq 0, L_{(x-yh_1,t-yh_2)}\ge  k+1}
$$
(with $\sup \varnothing =0$). If $\gamma >0$ we have
$$
k+1 = L_{(x-\gamma h_1,t-\gamma h_2)}=L_{(x-\gamma h_1,t-\gamma h_2)^-}+1.
$$
From \eqref{eq:Coupling} 
$$
k+1=L^{{\bf h}}_{(x-\gamma h_1+h_1 k,t-\gamma h_2+h_2 k)}=L^{{\bf h}}_{(x-\gamma h_1+h_1 k,t-\gamma h_2+h_2 k)^-}+1.
$$
By monotonicity of $L$ we deduce that
$$
\mathbb{P}( L^{{\bf h}}_{(x,t)}< k+1)=\mathbb{P}( \gamma < k).
$$
On the other hand, by definition of $\gamma$,
$$
\mathbb{P}( \gamma < k)=\mathbb{P}( L_{(x-h_1 k,t-h_2 k)}<k+1),
$$
and Theorem \ref{th:Coupling} is proved.



\end{proof}

\begin{remark}
There is a geometric interpretation of the coupling equality \eqref{eq:Coupling}: the image of a Hammersley line under mapping $\phi^{\bf h}$ is a Hammersley line as well. More precisely, for every $\ell$,
$$
\phi^{\bf h}\left(\mathcal{H}^{{\bf 0}}_\ell\right)=\mathcal{H}^{{\bf h}}_\ell,
$$
where on the left-hand side $\mathcal{H}^{{\bf 0}}_\ell$ is defined with the points of $\Xi$ and on the right-hand side, $\mathcal{H}^{{\bf h}}_\ell$ is defined with the points of $\tilde{\Xi}$. 
\end{remark}

\begin{proof}[Proof of Proposition \ref{TheoLLN_continu}]
Recall the asymptotics known for the length of the longest increasing path: 
\begin{equation*}
f(a,b):=\lim_{t\to +\infty}\frac{L_{( at,bt)}}{t}= 2\sqrt{ab}.
\end{equation*}
We fix $(a,b)\in(0,+\infty)^2$ , let $\lambda>0$ and $t$ such that $\lambda  t \in \Z_{\ge  0}$. Using Theorem \ref{th:Coupling} we obtain
$$
\mathbb{P}( L^{{\bf h}}_{(at,bt)}\le \lambda t)= \mathbb{P}( L_{(at-h_1 \lambda t,bt-h_2 \lambda t )}\le  \lambda t)
\stackrel{t\to +\infty}{\to}
\begin{cases}
0 &\text{ if }\lambda< f\left(a-h_1\lambda,b-h_2\lambda \right),\\
1 &\text{ if }\lambda> f\left(a-h_1\lambda,b-h_2\lambda \right).
\end{cases}
$$
Therefore, $\tfrac{1}{t}L^{{\bf h}}_{(at,bt)}$ converges in probability to the unique solution $\lambda$ of the equation 
\begin{equation}\label{Eq:def_fh}
\lambda=f\left(a-h_1\lambda,b-h_2\lambda \right) \quad \mbox{\emph{ i.e. }} \quad \lambda=2\sqrt{(a-h_1\lambda)(b-h_2\lambda)}.
\end{equation}
We easily check that if $h_1h_2\neq 1/4$ the solution of \eqref{Eq:def_fh} is given by 
$$
\lambda= \frac{2(ah_2+bh_1)-2\sqrt{(ah_2-bh_1)^2+ab}}{4h_1h_2-1}.
$$
If $h_1h_2=1/4$, then \eqref{Eq:def_fh} reduces to $\lambda=ab/(h_1b+h_2a)$.
\end{proof}

\subsection{Fluctuations of $L^{\bf h}(at,bt)$}
Let us now explain how the combination of Baik-Deift-Johansson's result (Theorem \ref{Prop:Flu_Ln_continu}) and Theorem \ref{th:Coupling} implies Proposition \ref{Theoflu_continu} for the fluctuations of $L^{\bf h}(at,bt)$.
\begin{proof}[Proof of Proposition \ref{Theoflu_continu}]
Using the scaling invariance of a Poisson point process under transformations which preserve the volume, in the case without gaps constraint, the distribution of $L_{(x,t)}$ only depends on the value of the product $xt$.
Thus,  we can define a family of random variables $(Z(s),s\ge 0)$ such that $Z(xt)\overset{d}{=}L_{(x,t)}$ for all $x,t\ge 0$. Theorem \ref{Prop:Flu_Ln_continu} yields that 
$$\forall c\in \R, \quad \lim_{s\to \infty}\P(Z(s^2)\le 2s+cs^{1/3})=F_{TW}(c),$$
where $F_{TW}$ is the distribution function of the Tracy-Widom distribution.

Fix now $a,b>0$ and let $\lambda=f^{\bf{h}}(a,b)$. 
Using Theorem \ref{th:Coupling}, we have that, for any $t\ge 0$ and $\beta\in \R$ such that $\lambda t+\beta t^{1/3}\in \Z_{\ge 0}$,  
\begin{eqnarray}\label{Eq:calcul_fluc}
\mathbb{P}( L^{{\bf h}}_{(at,bt)}\le\lambda t+\beta t^{1/3})&=& \mathbb{P}( L_{(t(a-h_1\lambda)-\beta h_1 t^{1/3},t(b-h_2\lambda)-\beta h_2 t^{1/3})}\le\lambda t+\beta t^{1/3}).\\ \nonumber
&=& \mathbb{P}(Z(s^2)\le \lambda t+\beta t^{1/3})
\end{eqnarray}
with $s\ge 0$ defined by
\begin{eqnarray*}
s^2&:=&(t(a-h_1\lambda)-\beta h_1 t^{1/3})(t(b-h_2\lambda)-\beta h_2 t^{1/3})\\
&=& t^2(a-\lambda h_1)(b-\lambda h_2)-t^{4/3}\beta (h_1b+h_2a-2\lambda h_1h_2)+ \mathcal{O}(t)\\
&=& t^2\frac{\lambda^2}{4}-t^{4/3}\beta (h_1b+h_2a-2\lambda h_1h_2)+ \mathcal{O}(t),
\end{eqnarray*}
where we use \eqref{Eq:def_fh} in the last line. Inverting this equality gives
\begin{equation*}
t=\frac{2}{\lambda}s+\beta \delta s^{1/3}+ \mathcal{O}(1) 
\quad \mbox{ with } \quad
\delta:=2^{4/3}\frac{h_1b+h_2a-2\lambda h_1 h_2}{\lambda^{7/3}}.
\end{equation*}
Plugging this expression of $t$  in \eqref{Eq:calcul_fluc}, we get 
\begin{eqnarray*}
\mathbb{P}( L^{{\bf h}}_{(at,bt)}\le \lambda t+\beta t^{1/3})
&=& \mathbb{P}\left(Z(s^2)\le 2s+\beta \Big(\lambda\delta+\frac{2^{1/3}}{\lambda^{1/3}}\Big) s^{1/3}+\mathcal{O}(1)\right).
\end{eqnarray*}
If we set $\sigma^{\bf h}(a,b)=\big(\lambda\delta+\frac{2^{1/3}}{\lambda^{1/3}}\big)^{-1}$ and apply the above equation with $\beta=c \sigma^{\bf h}(a,b)$, we obtain
$$\lim_{t\to \infty}\P(L^{{\bf h}}_{(at,bt)}\le \lambda t+c\sigma^{\bf h}(a,b) t^{1/3})=
\lim_{s\to \infty}\P(Z(s^2)\le 2s+cs^{1/3})=F_{TW}(c).$$
One can  check that this definition of  $\sigma^{\bf h}(a,b)$ coincides with the one given in Proposition \ref{Theoflu_continu}.
\end{proof}

\subsection{Case where $h$,$\lambda$ depend on $t$}\label{Sec:EcartQuiBouge}

In this short section we show how to use the scale-invariance of the Poisson point process to derive asymptotics in the case where gaps and intensity of points depend on $t$. For the sake of simplicity we assume that vertical and horizontal gaps are identical.

For every $t$, let ${\bf h_t}=(h_t,h_t)$ be a pair of gaps, $\lambda_t>0$ and denote by $L^{{\bf h_t},\lambda_t}_{(t,t)}$ be the length of the longest increasing path with gaps ${\bf h_t}$ when $\Xi$ is a Poisson process with intensity $\lambda_t$.

\begin{theo}\label{TheoToutDependDeT2}Let $c_t=h_t\sqrt{\lambda_t}$ and assume that $c:=\lim_{t \to \infty} c_t$ exists in $[0,+\infty]$. Then, 
\begin{itemize}
\item[(i)] If $c=0$ and  $\sqrt{\lambda_t}t\to +\infty$ then
$$
\frac{1}{\sqrt{\lambda_t}t} L^{{\bf h_t},\lambda_t}_{(at,bt)} \stackrel{\text{prob.}}{\to} f^{(0,0)}(a,b).
$$
\item[(ii)] If  $c\in (0,+\infty)$ and  $\sqrt{\lambda_t}t\to +\infty$  then
$$
\frac{1}{\sqrt{\lambda_t}t} L^{{\bf h_t},\lambda_t}_{(at,bt)} \stackrel{\text{prob.}}{\to} f^{(c,c)}(a,b).
$$
\item[(iii)] If $c=+\infty$ and  $t/h_t \to +\infty$  then
$$
\frac{h_t}{t}L^{{\bf h_t},\lambda_t}_{(at,bt)} \stackrel{\text{prob.}}{\to} \min\{a,b\}.
$$
\end{itemize}
\end{theo}

\begin{proof}
We observe that by the scaling invariance of the Poisson process we have
$$
L^{{\bf h_t},\lambda_t}_{(at,bt)} \stackrel{\text{(d)}}{=} L^{{\bf h_t\sqrt{\lambda_t}},1}_{(at\sqrt{\lambda_t},bt\sqrt{\lambda_t})}.
$$
Assume first that $c=\lim_{t \to \infty} h_t\sqrt{\lambda_t} <+\infty$. 
Let $\varepsilon\in (0,1)$ and $T$ such that $c(1-\varepsilon)\le h_t\sqrt{\lambda_t}\le c(1+\varepsilon)$ for $t\ge T$. Since $L^{\bf h}_{(at,bt)}$ is a non-increasing function in $\bf h$, we get, for $t\ge T$, the stochastic domination:
$$\frac{1}{t\sqrt{\lambda_t}}L^{{\bf c(1+\varepsilon)},1}_{(at\sqrt{\lambda_t},bt\sqrt{\lambda_t})} \preccurlyeq  \frac{1}{t\sqrt{\lambda_t}}L^{{\bf h_t},\lambda_t}_{(at,bt)} \preccurlyeq 
\frac{1}{t\sqrt{\lambda_t}}L^{{ \bf c(1-\varepsilon)},1}_{(at\sqrt{\lambda_t},bt\sqrt{\lambda_t})} 
$$
Assuming that $t\sqrt{\lambda_t}$ tends to infinity, the left-hand side tends to $f^{(c(1+\varepsilon),c(1+\varepsilon))}(a,b)$
whereas the right-hand side tends to $f^{(c(1-\varepsilon),c(1-\varepsilon))}(a,b)$. We conclude by continuity in $c$ of the expression of $f^{(c,c)}(a,b)$. 

Assume now that $c=\lim_{t \to \infty} h_t\sqrt{\lambda_t}=+\infty$. First, by definition of gaps, $L_{(at,bt)}^{{\bf h_t},\lambda}\le  \min \{a,b\}\times t/h_t$ a.s. This gives the upper bound in (iii).\\
For the lower bound, let $A>0$ and $T$ such that $ h_t\sqrt{\lambda_t}\ge A$ for $t\ge T$. For $t\ge T$, 
$\lambda_t\ge (A/h_t)^2$, thus using the monotonicity of a Poisson point process with respect to its intensity, we have
\begin{equation}\label{eq:MinoQuandCaDepend}
\frac{h_t}{t}L^{{\bf h_t},\lambda_t}_{(at,bt)} \succcurlyeq 
\frac{h_t}{t}L^{{\bf h_t},(A/h_t)^2}_{(at,bt)}.
\end{equation}
Using (ii)  with $\tilde{\lambda}_t=(A/h_t)^2$, we get, if $t/h_t\to +\infty$, the following convergence in probability:
$$
\lim_{t\to \infty} \frac{h_t}{At}L^{{\bf h_t},(A/h_t)^2}_{(t,t)} = f^{(A,A)}(a,b).
$$
Observe that $\lim_{A\to +\infty} Af^{(A,A)}(a,b) =\frac{a+b-|a-b|}{2}= \min \{a,b\}$, so we obtain the lower bound by letting $A$ tend to infinity in \eqref{eq:MinoQuandCaDepend}.

\end{proof}

\section{Proofs in the discrete settings}

\begin{figure}
\begin{center}
\includegraphics[width=10cm]{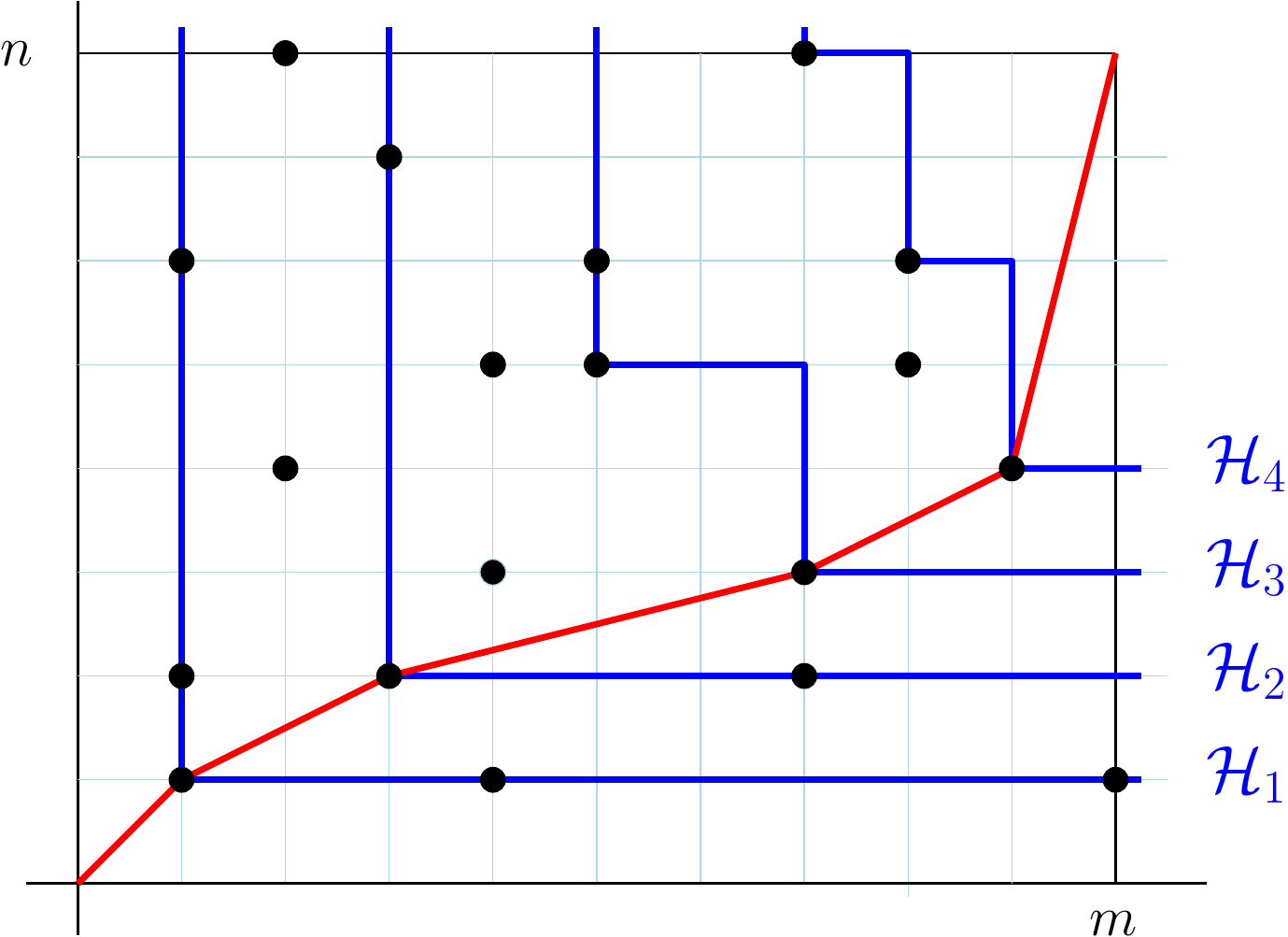}
\end{center}
\caption{An example with ${\bf h}=(2,1)$ for the same realization of $\Xi$ as that of Fig.\ref{Fig:SansChemin}. In every unit square $(i,j)$ we write the value of $\Lr^{(2,1)}_{(i,j)}$, the four Hammersley lines are drawn in blue. }
\label{Fig:DefinitionLignesHammersley_Gap2}
\end{figure}

\subsection{Proof of Theorem \ref{th:DiscreteCoupling}}

The main task of this Section is to prove the coupling between $\Lr^{{\bf h}}$ and $T$ which leads to the identity of Theorem \ref{th:DiscreteCoupling}. For the sake of clarity we first exhibit a coupling between $\Lr^{{\bf h}}$ and $\Lr^{(1,1)}$. 
\begin{lem}\label{lem:Coupling_h_11}
Let ${\bf h}=(h_1,h_2)$ with $(h_1,h_2)\neq (0,0)$.
For every $m,n\geq 0$, and every $k \geq 0$,
$$
\mathbb{P}( \Lr^{{\bf h}}_{(m,n)}\le k)= \mathbb{P}( \Lr^{(1,1)}_{(m-(h_1-1)k,n-(h_2-1)k)}\le k).
$$
\end{lem}
\subsubsection{Proof of Lemma \ref{lem:Coupling_h_11}: the case $h_1>0,h_2>0$}

In the discrete settings and if $h_1>0,h_2>0$, the proof of Lemma \ref{lem:Coupling_h_11} is almost identical to that of Theorem \ref{th:Coupling}. We only explain how to change the definitions of the Hammersley lines and the function $\phi^{\bf h}$.

The definition of the Hammersley lines is identical to the continuous case (an example is provided in Fig.\ref{Fig:DefinitionLignesHammersley_Gap2}): the broken line $\mathcal{H}^{\bf h}_1$ is the shortest path made of vertical and horizontal straight lines whose minimal points for $\stackrel{{\bf 0}}{\prec}$ are exactly  the minimal points of $\Xi$ for  $\stackrel{{\bf 0}}{\prec}$.
The  line $\mathcal{H}^{\bf h}_2$ is defined as follows: we remove the points of $\mathcal{H}_1^{\bf h}+\set{0,1,\dots,h_1-1}\times\set{0,1,\dots,h_2-1}$ and reiterate the procedure: $\mathcal{H}^{\bf h}_2$ is the shortest path made of vertical and horizontal straight lines whose minimal points for $\stackrel{{\bf 0}}{\prec}$ are exactly  the minimal points of $\Xi\setminus \left(\mathcal{H}^{\bf h}_1+\set{0,1,\dots,h_1-1}\times\set{0,1,\dots,h_2-1}\right)$  for  $\stackrel{{\bf 0}}{\prec}$.
Inductively we define $\mathcal{H}^{\bf h}_3,\mathcal{H}^{\bf h}_4,\dots$ in the same way.

The function $\phi^{\bf h}$ has to be replaced by its discrete counterpart:
\begin{equation}\label{eq:phidiscret}
\begin{array}{r c c c}
\phi^{\bf h}: & (\bbZ_{>0})^2 &   \to   &    (\bbZ_{>0})^2 \\
      & (m,n)         & \mapsto & (m+(h_1-1)\Lr^{(1,1)}_{(m-1,n-1)},n+(h_2-1)\Lr^{(1,1)}_{(m-1,n-1)}). 
\end{array}
\end{equation}
We define a new set of points $\tilde{\Xi}$ by 
$$
\begin{cases}
\tilde{\Xi}_{i',j'}&= \Xi_{i,j} \text{ if }(i',j')\in\mathrm{Image}(\phi^{\bf h}) \text{ and }\phi^{\bf h}(i,j)=(i',j'),\\
\tilde{\Xi}_{i',j'}&= \mathbf{Y}_{i',j'} \text{ if }(i',j')\notin\mathrm{Image}(\phi^{\bf h}),
\end{cases}
$$
where $(\mathbf{Y}_{i',j'})_{i,j\ge 1}$ are independent Bernoulli random variables with mean $p$ (see an example in Figure \ref{Fig:FonctionPhi3-2}).
In the same manner as in the continuous settings, we prove that  $(\tilde{\Xi}_{i,j})_{i,j\ge 1}$ are i.i.d Bernoulli random variables with mean $p$ and 
for every $m,n\ge 1$ we have
$$
\Lr^{(1,1)}_{(m,n)}(\Xi)= \Lr^{{\bf h}}_{(m',n')}(\tilde{\Xi}),
$$
where $(m',n')=\phi^{ {\bf h}}(m,n)=(m+(h_1-1)\Lr^{(1,1)}_{(m-1,n-1)},n+(h_2-1)\Lr^{(1,1)}_{(m-1,n-1)})$.
We deduce then Lemma \ref{lem:Coupling_h_11} for $h_1h_2>0$ in the same manner as in  Theorem \ref{th:Coupling}.

\begin{figure}
\begin{center}
\includegraphics[width=14cm]{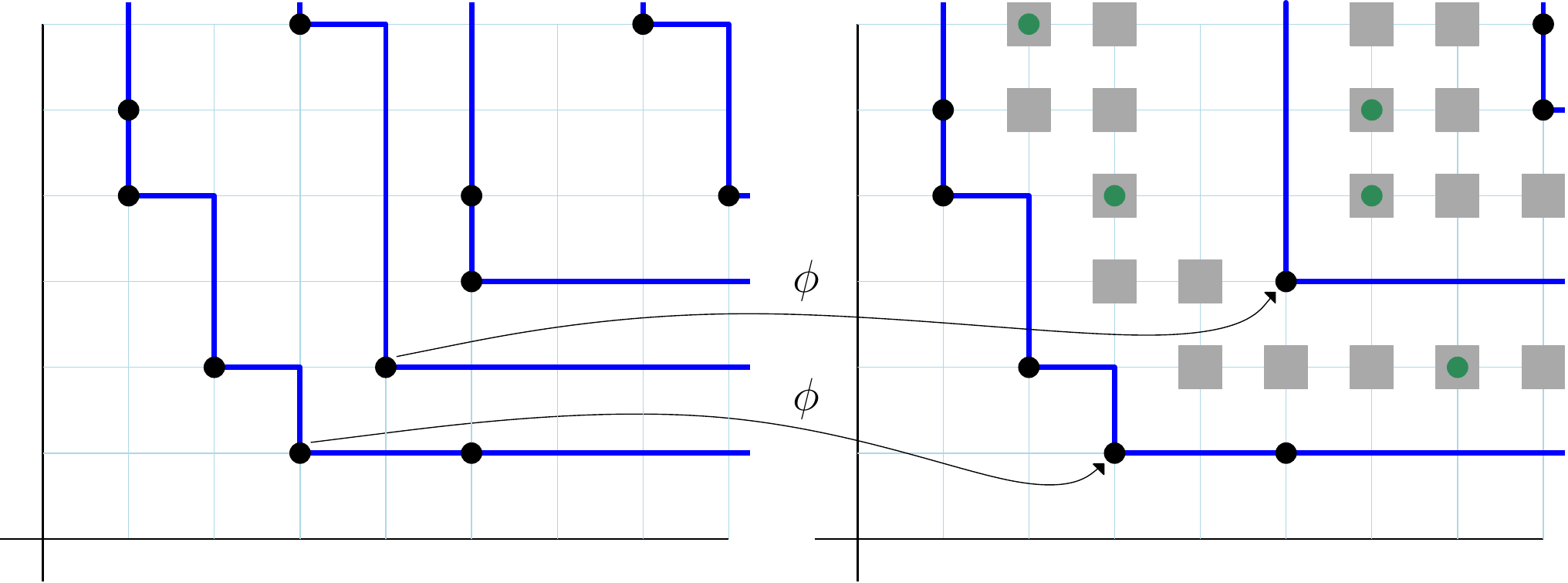}
\end{center}
\caption{An example of the function $\phi^{\bf h}$ for $(m,n)=(8,6)$ and ${\bf h}=(3,2)$. The points of the quarter-plane indicated by small gray squares are not in the image of $\phi^{\bf h}$. The values of $\tilde{\Xi}$ at these points are independent of $\Xi$. 
}
\label{Fig:FonctionPhi3-2}
\end{figure}


\subsubsection{Proof of Lemma \ref{lem:Coupling_h_11}: the case $h_1>0,h_2=0$}

As in the previous section, we can exhibit a coupling between  $\Lr^{(h_1,0)}$ and $\Lr^{(1,1)}$ which  shows that Lemma \ref{lem:Coupling_h_11} also holds for ${\bf h}=(h_1,0)$.
However, some change must be made compared to the case $h_1 h_2 >0$ since the function 
 $\phi^{\bf h}$ defined in \eqref{eq:phidiscret} is no more a dilation. 

To make the exposition clearer, it is more convenient to explain the coupling between the model with gap $(h,0)$ with the one with gap $(h,1)$. Thus, let us consider a Bernoulli  field $\Xi$ on  $(\bbZ_{>0})^2$ and construct the associated random variables $\Lr^{(h,1)}_{(m,n)}$.
Define the function $\psi^{ h}$ by
$$
\begin{array}{r c c c}
\psi^{h}: & (\bbZ_{>0})^2 &   \to   &    (\bbZ_{>0})^2 \\
      & (m,n)         & \mapsto & (m,n-\Lr^{(h,1)}_{(m-1,n-1)}). 
\end{array}
$$
Contrary to $\phi^{\bf h}$, the function $\psi^{ h}$ is surjective but no more injective. More precisely, for any $(m,n')\in  (\bbZ_{>0})^2 $, there exist $n\ge 1$ and $k\ge 0$ such that 
$$(\psi^{ h})^{-1}(m,n')=\{(m,n),(m,n+1),\ldots,(m,n+k)\}.$$
We define now the new set of points $\tilde{\Xi}$ by
$$\tilde{\Xi}_{m,n'}=\Xi_{m,n+k} \mbox{ where } (\psi^{ h})^{-1}(m,n')=\{(m,n),(m,n+1),\ldots,(m,n+k)\}.$$

\begin{figure}
\begin{center}
\includegraphics[width=14cm]{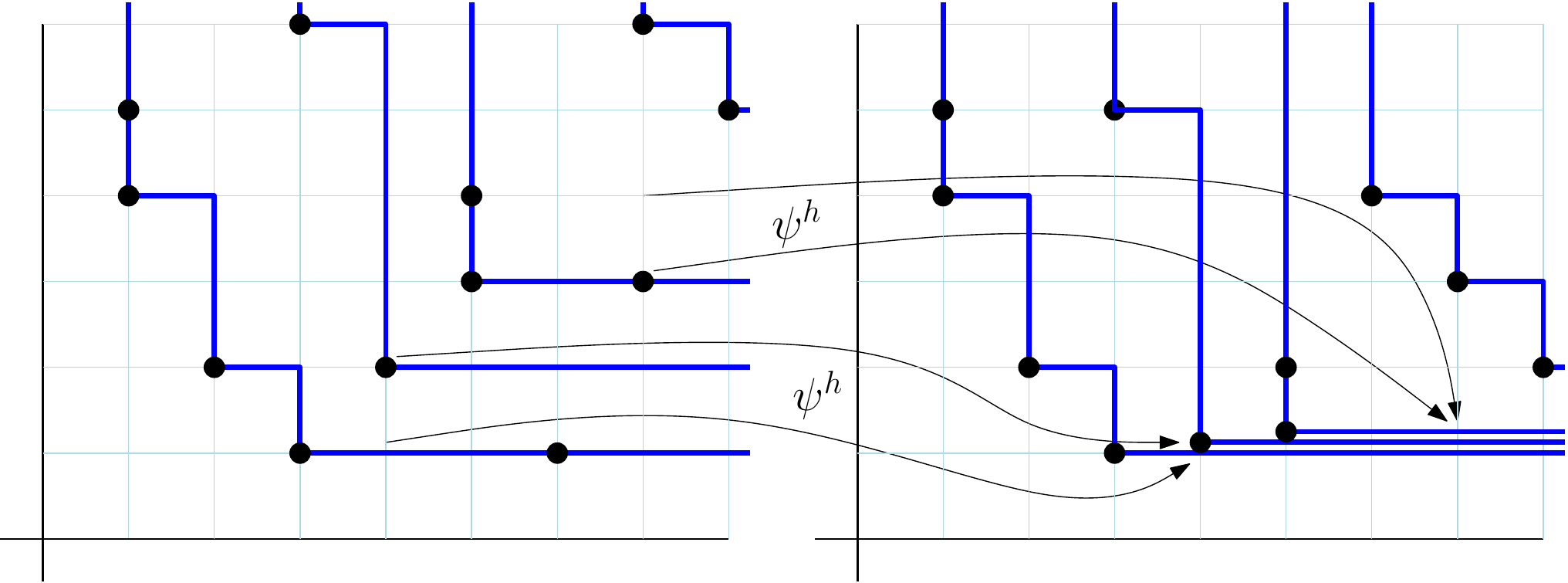}
\end{center}
\caption{An example of the function $\psi^{ h}$ and the definition of set of points $\tilde{\Xi}$ for ${\bf h}=(1,0)$.  There is a point in $\tilde{\Xi}$ at $(i',j')$ i.f.f. there is one in  $\Xi$ at its highest antecedent by $\psi^h$. 
}
\label{Fig:FonctionPhi3-0}
\end{figure}
Again, one can prove that the random variables $\tilde{\Xi}=(\tilde{\Xi}_{i,j})_{i,j\ge 1}$ are i.i.d. Bernoulli random variables with mean $p$ and 
for every $m,n\ge 1$ we have
$$
\Lr_{(m,n)}^{(h,1)}(\Xi)= \Lr^{(h,0)}_{(m',n')}(\tilde{\Xi}),
$$
where $(m,n')=\psi^{ h}(m,n)=(m,n-\Lr^{(h,1)}_{(m-1,n-1)})$.
Then, with the same argument as in the proof of Theorem \ref{th:Coupling}, we get that, for every $m,n \in \bbZ_{\geq 0}$, and every $k \in \bbZ_{\geq 0}$,
$$
\mathbb{P}( \Lr^{(h,0)}_{(m,n)}\le k)= \mathbb{P}( \Lr^{(h,1)}_{(m,n+k)}\le k)= \mathbb{P}( \Lr^{(1,1)}_{(m-(h-1)k,n+k)}\le k).
$$

\subsubsection{Proof of Theorem \ref{th:DiscreteCoupling}: coupling with $T$}\label{Sec:CouplageLPP}

We conclude the proof of Theorem \ref{th:DiscreteCoupling} with our last coupling between $\Lr_{(m,n)}^{(1,1)}$ and $T_{(m',n')}$, for some $(m',n')$.
As already said, this coupling already appeared in \cite{Dhar98,GeorgiouOrtmann,MajumdarNechaevEARMSA}.

Let us consider a Bernoulli  field $\Xi$ on  $(\bbZ_{>0})^2$ and construct the associated random variables $\Lr^{(1,1)}_{(m,n)}$ associated to the gaps $(1,1)$. Formally, in the case ${\bf h}=(0,0)$, the function $\phi^{\bf h}$ defined in \eqref{eq:phidiscret} becomes
\begin{equation}\label{eq:phidiscret2}
\begin{array}{r c c c}
\phi^{\bf 0}: & (\bbZ_{>0})^2 &   \to   &    (\bbZ_{>0})^2 \\
      & (m,n)         & \mapsto & (m-\Lr_{(m-1,n-1)},n-\Lr_{(m-1,n-1)}). 
\end{array}
\end{equation}
As in the previous case ${\bf h}=(h,0)$, the function $\phi^{\bf 0}$ is surjective but not injective. More precisely, for any $(m',n')\in  (\bbZ_{>0})^2 $, there exist $n,m\ge 1$ and $k\ge 0$ such that 
$$(\phi^{\bf 0})^{-1}(m',n')=\{(m,n),(m+1,n+1),\ldots,(m+k,n+k)\}.$$
We first define a new collection of  random variables $\hat{\Xi}=\{\hat{\Xi}_{i,j},i,j\ge 1\}\in \{0,1\}^{(\Z_{> 0})^2}$ by
$$
\begin{cases}
\hat{\Xi}_{m,n}&= \Xi_{m,n}=1 \text{ if } (m,n) \mbox{ is a minimal point of some } \mathcal{H}^{(1,1)}_\ell \\
\hat{\Xi}_{m,n}&= 0 \text{ otherwise,}
\end{cases}
$$
and 
we define now the family  of random variables $\tilde{\Xi}=\{\tilde{\Xi}_{i,j},i,j\ge 1\}\in (\Z_{\ge 0})^{(\Z_{> 0})^2}$ by
$$\tilde{\Xi}_{m',n'}=\sum_{(m,n)\in (\phi^{\bf 0})^{-1}(m',n')}\hat{\Xi}_{m,n}.$$ 
\begin{figure}
\begin{center}
\includegraphics[width=14cm]{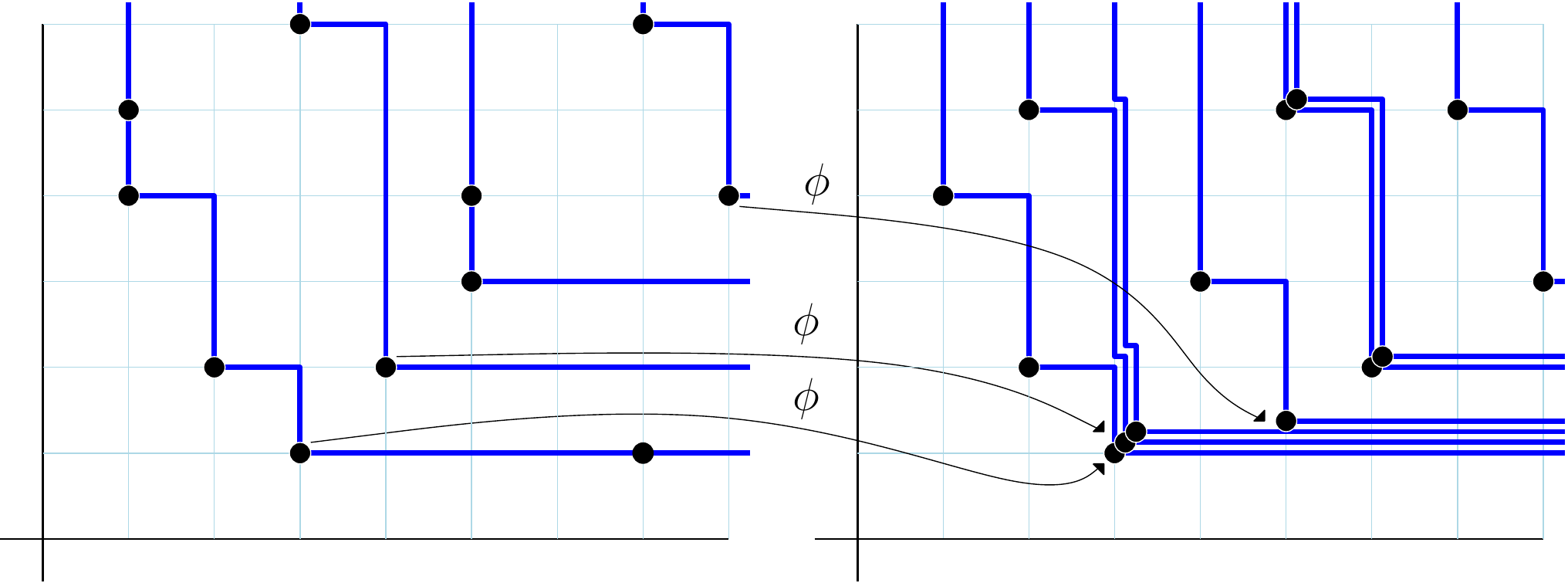}
\end{center}
\caption{An example of the function $\phi^{\bf 0}$. Left: a realization of Hammersley lines $\mathcal{H}^{(1,1)}_\ell$. Right: The associated realization of last-passage percolation with geometric weights. We have $\tilde{\Xi}_{3,1}=\hat{\Xi}_{3,1}+\hat{\Xi}_{4,2}+\hat{\Xi}_{5,3}+\hat{\Xi}_{6,4}=1+1+1+0=3$.
}
\label{Fig:FonctionPhi0-0}
\end{figure}
For every $m',n'$ we have $\tilde{\Xi}_{(m',n')}=k$ if 
$$
\hat{\Xi}_{(m,n)}=\hat{\Xi}_{(m+1,n+1)}= \dots =\hat{\Xi}_{(m+k-1,n+k-1)}=1,\quad \hat{\Xi}_{(m+k,n+k)}=0,
$$
which occurs with probability $p^k(1-p)$. Therefore we can show that $\tilde{\Xi}$ is a family of i.i.d. geometric random variables:
$\P(\tilde{\Xi}_{m,n}=k)=p^k(1-p)$ for $k\ge 0$.

For $m,n\ge 1$, recall the notation
\begin{equation*}
T_{(m,n)}(\tilde{\Xi})= 
\max \{ \sum_{(i,j)\in P} \tilde{\Xi}_{i,j} \; ; \; P\in \mathcal{P}_{m,n}\},
\end{equation*}
where $\mathcal{P}_{m,n}$ denotes the set of paths from $(1,1)$ to $(m,n)$ taking only North and East steps. 
With the same arguments of the previous cases one can prove that for every $m,n\ge 1$ we have
$$
\Lr_{(m,n)}^{(1,1)}(\Xi)= T_{(m',n')}(\tilde{\Xi}),
$$
where $(m',n')=\phi^{{\bf 0}}(m,n)=(m-\Lr^{(1,1)}_{(m-1,n-1)},n-\Lr^{(1,1)}_{(m-1,n-1)})$.

We deduce that for every $m,n \ge 0$, and every $k \ge 0$,
$$
\mathbb{P}( T_{(m,n)}\le k)= \mathbb{P}( \Lr^{(1,1)}_{(m+k,n+k)}\le k),
$$
which is  in fact equivalent to Eq.(4.1) in \cite{GeorgiouOrtmann}. Combining this equality with Lemma \ref{lem:Coupling_h_11} yields Theorem \ref{th:DiscreteCoupling}.


\subsection{Proof of the limiting shape: Proposition \ref{TheoLLN}}

\begin{proof}[Proof of Proposition \ref{TheoLLN}]
We fix $(a,b)\in(0,+\infty)^2$ , let $\lambda>0$ and $n$ such that $\lambda  n \in \Z_{\ge  0}$ and $\lambda\le \lambda_0:= \min(a/h_1,b/h_2)$. This last condition implies in particular that $(an-h_1\lambda n)$ and $(bn-h_2\lambda n)$ are non negative. Note also that, due to the gap constraint, we have $\Lr^{{\bf h}}_{(an,bn)}\le \lambda_0n$ a.s.
  Combining Theorem \ref{th:DiscreteCoupling} and Theorem  \ref{Th:JockuschProppShor}, we have
\begin{align*}
\mathbb{P}( \Lr^{{\bf h}}_{(an,bn)}\le\lambda n)= &\ \mathbb{P}( T_{(an-h_1\lambda n,bn-h_2\lambda n)} \le \lambda n)
\stackrel{n\to +\infty}{\to}
&\begin{cases}
0 &\text{ if }\lambda< g\left(a-h_1\lambda,b-h_2\lambda \right),\\
1 &\text{ if }\lambda> g\left(a-h_1\lambda,b-h_2\lambda \right)
\end{cases}
\end{align*}
where $g$ is defined in \eqref{eq:gTASEP}.
Therefore, $\tfrac{1}{n}\Lr^{{\bf h}}_{(an,bn)}$ converges in probability to
$$g^{\bf h}(a,b):=\sup\{\lambda\le \lambda_0, \lambda< g\left(a-h_1\lambda,b-h_2\lambda \right)\}.$$
Note that in \eqref{eq:gTASEP}, $g$ is only defined on $(\R_+^*)^2$ but one can  extend $g$ on $(\R_+)^2$ by continuity so that $g\left(a-h_1\lambda_0,b-h_2\lambda_0 \right)$ is well defined.
Two cases can occur: 
\begin{itemize}
\item either 
\begin{equation*}
\lambda_0\ge g\left(a-h_1\lambda_0,b-h_2\lambda_0 \right),
\end{equation*}
and  the equation 
\begin{equation}\label{Eq:def_fh_discret}
\lambda= g\left(a-h_1\lambda,b-h_2\lambda \right),
\end{equation}
has a solution which is necessarily unique since the right hand side of  \eqref{Eq:def_fh_discret} decreases with respect to $\lambda$. Then $\tfrac{1}{n}\Lr^{{\bf h}}_{(an,bn)}$ converges to this unique solution.
\item Or 
\begin{equation*}
\lambda_0< g\left(a-h_1\lambda_0,b-h_2\lambda_0 \right),
\end{equation*}
and in this case,  $\tfrac{1}{n}\Lr^{{\bf h}}_{(an,bn)}$ converges $\lambda_0=\min(a/h_1,b/h_2).$
\end{itemize}
Using the expression of $g$ given in \eqref{eq:gTASEP}, one can check that $\lambda_0\ge g\left(a-h_1\lambda_0,b-h_2\lambda_0 \right)$ i.f.f. 
 $\frac{h_1}{(h_2-1)+1/p}\le\frac{a}{b}\le\frac{h_1-1+1/p}{h_2}$
\end{proof}
\subsection{Fluctuations of $\Lr_{(an,bn)}^{\bf h}$}
%
Johansson \cite{Johansson1} has computed the fluctuations of $T_{(\lf an\rf,\lf bn\rf)}$ around its mean:
\begin{theo}[Cube root fluctuations (Johansson \cite{Johansson1}, Theorem 1.2)]\label{Prop:FluctuationsJohansson}
For every $a,b>0$ and $x\in \R$, we have 
\begin{equation*}
\lim_{n\to \infty}\mathbb{P}\left(\frac{T_{(\lf an\rf,\lf bn\rf)} -ng(a,b)}{\sigma(a,b)n^{1/3}}\le x \right)
=F_{TW}(x), 
\end{equation*}
where $F_{TW}$ is the distribution function of the Tracy-Widom distribution,
and
$$\sigma(a,b)=\frac{p^{1/6}}{1-p}(ab)^{-1/6} (\sqrt{a}+\sqrt{pb})^{2/3}(\sqrt{b}+\sqrt{pa})^{2/3}.$$
\end{theo}

In another paper Johansson \cite[Th.5.3]{Johansson} has also obtained Tracy-Widom fluctuations for longest increasing paths in the case ${\bf h}=(1,0)$. 
The authors of \cite{MajumdarNechaevEARMSA} state a close result for the fluctuations of $\Lr_{(an,bn)}$ around its mean (see also Section 4 in \cite{Prieehev}). However, we have not been able to fill the gap between their result and the convergence of rescaled fluctuations.

From Theorem \ref{Prop:FluctuationsJohansson}, it is not obvious to obtain a result as neat as Proposition \ref{Theoflu_continu} for every direction $(a,b)$. The proof of Proposiion \ref{Theoflu_continu} relies on the scaling invariance property of the Poisson process: the law of $L_{(x,t)}$ only depends on the value of $xt$. There is of course no analogous for fields of Bernoulli random points. However, one can still show that for any $a,b\ge 0$ and ${\bf h}$, the fluctuations of $\Lr^{\bf h}_{(\lf an\rf,\lf bn\rf)}$ are also of order $n^{1/3}$ (outside the flat edges of the limiting shape). Before stating our result about the fluctuations of $\Lr^{\bf h}_{(\lf an\rf,\lf bn\rf)}$, we must first prove a technical lemma.

\begin{lem}\label{lem:technique}
Let $\bf{h}$ be a gap constraint. 
For all $a,b> 0$ such that $g^{\bf h}(a,b)<\min\{a/h_1,b/h_2\}$, there exists a unique couple $(\alpha,\beta)$ of positive numbers such that $g^{\bf h}(a,b)=g(\alpha,\beta).$ Moreover, $(\alpha,\beta)$ is solution   to the system
$$
\begin{cases}
&\alpha+h_1g(\alpha,\beta)=a,\\
&\beta+h_2g(\alpha,\beta)=b.
\end{cases}
$$
\end{lem}

\begin{proof}By symmetry, we can assume that $g^{\bf h}(a,b)<a/h_1\le b/h_2$.
Recall that
 in this case, $g^{\bf h}(a,b)$ is the unique $\lambda$ solution of 
 \begin{equation}\label{eq:lemtech}
\lambda= g\left(a-h_1\lambda,b-h_2\lambda \right),
\end{equation}
and we necessarily have  
$ g\left(0,b-\frac{h_2a}{h_1}\right)<\frac{a}{h_1}$.

Assume that there exists a solution $(\alpha,\beta)$ of the system
$$\alpha+h_1g(\alpha,\beta)=a,\qquad \beta+h_2g(\alpha,\beta)=b.$$
Putting $g(\alpha,\beta)$ in \eqref{eq:lemtech}, we see that $g(\alpha,\beta)$ satisfies this equality and thus $g(\alpha,\beta)=g^{\bf h}(a,b)<a/h_1$. In particular, we necessarily have $\alpha, \beta>0$. It remains to prove that the system as indeed a (unique) solution.
Noticing that $g(\alpha,\beta)=\beta g(\alpha/\beta,1)$ and setting $\gamma=\alpha/\beta$, we see now that the system is equivalent to 
\begin{equation}\label{Eq:systemred}
\beta(\gamma+h_1g(\gamma,1))=a,\qquad \beta(1+h_2g(\gamma,1))=b.
\end{equation}
In particular, we have 
$$b(\gamma+h_1g(\gamma,1))=a(1+h_2g(\gamma,1)).$$
Hence, if $a/h_1=b/h_2$, we get $\gamma=a/b$.
In the other case : $a/h_1< b/h_2$, we get
$$g(\gamma,1)=\frac{a-b\gamma}{bh_1-ah_2}.$$
The left hand side is increasing with $\gamma$ whereas the right hand side decreases. Thus, there exists a unique solution $\gamma>0$ if and only if
$$g(0,1)<\frac{a}{bh_1-ah_2}$$
which coincides with the  condition $ g\left(0,b-\frac{h_2a}{h_1}\right)<\frac{a}{h_1}$ stated above. Finally, using \eqref{Eq:systemred}, we see that
the existence and unicity of $\gamma$ implies the existence and unicity of $(\alpha,\beta)$.
\end{proof}

\begin{prop}\label{Prop:FluctuationsDiscretes}
Let ${\bf h}=(h_1,h_2)$ be a gap constraint and let
 $a,b> 0$ be such that $g^{\bf h}(a,b)<\min\{a/h_1,b/h_2\}$. Let us define $(\alpha,\beta)$ as in Lemma \ref{lem:technique} such that  $g^{\bf h}(a,b)=g(\alpha,\beta)$. Set 
$$\mathcal{W}^{\bf h}_{an,bn}=\frac{\Lr^{\bf h}_{(\lf an\rf,\lf bn\rf)}- g^{\bf h}(a,b)n}{ \sigma(\alpha,\beta)n^{1/3}}.$$
 Assuming for example that $a/h_1\le b/h_2$,
we have, for all $x\ge 0$, 
\begin{equation*}
F_{TW}\left(\frac{bx}{\beta}\right)\le \liminf_{n\to \infty} \mathbb{P}\left(\mathcal{W}^{\bf h}_{an,bn}\le x\right)\le \limsup_{n\to \infty}\mathbb{P}\left(\mathcal{W}^{\bf h}_{an,bn}\le x\right) \le F_{TW}\left(\frac{ax}{\alpha}\right)
\end{equation*}
\begin{equation*}
F_{TW}\left(\frac{-ax}{\alpha}\right)\le \liminf_{n\to \infty} \mathbb{P}\left(\mathcal{W}^{\bf h}_{an,bn}\le -x\right)\le \limsup_{n\to \infty}\mathbb{P}\left(\mathcal{W}^{\bf h}_{an,bn}\le -x\right) \le F_{TW}\left(\frac{-bx}{\beta}\right).
\end{equation*}
\end{prop}
In one particular direction the LHS and RHS of above inequalities coincide:
\begin{coro}
Let ${\bf h}=(h_1,h_2)$ be a gap constraint and let $a,b> 0$ be such that $a/h_1= b/h_2$. Then 
\begin{equation}\label{eq:Flu_dis}
\lim_{n\to \infty}\P\left(\frac{\Lr^{\bf h}_{(an,bn)} -g^{{\bf h}}(a,b)n}{\sigma^{\bf h}(a,b)n^{1/3}}\le x\right)=F_{TW}(x),
\end{equation}
where 
$$\sigma^{\bf h}(a,b)= \sigma(\alpha,\beta)\sqrt{\frac{\alpha\beta}{ab}}.$$
\end{coro}
Again, we have a simple expression for ${\bf h}=(h,h)$ and $a=b=1$:
$$
\sigma^{(h,h)}(1,1)= \frac{(1-p)^{1/3}p^{1/6}}{(1+(2h-1)\sqrt{p})^{4/3}}.
$$

\begin{remark}Proposition \ref{Prop:FluctuationsDiscretes} states that, ouside the flat edge of the limiting shape, the fluctuations of $\Lr^{\bf h}_{(an,bn)}$ are of order $n^{1/3}$. Inside the flat edge (except in the critical direction), one can easily show that $\P( |\Lr^{\bf h}_{(\lf an\rf,\lf bn\rf)}- g^{\bf h}(a,b)n|\ge 1)$ tends to 0.
\end{remark}

\begin{proof}
A change of indexes in Theorem \ref{th:DiscreteCoupling} yields
\begin{equation}\label{Eq:coupling0}
\mathbb{P}( T_{(m,n)}\le k)=  \mathbb{P}( \Lr^{\bf h}_{(m+h_1 k,n+h_2 k)}\le k).
\end{equation}
Let $a,b>0$ such that $g^{\bf h}(a,b)<\min\{a/h_1,b/h_2\}$ and, according to Lemma \ref{lem:technique}, take $\alpha,\beta>0$ solution of the system
$$\alpha+h_1g(\alpha,\beta)=a,\qquad \beta+h_1g(\alpha,\beta)=b$$
and such that $g^{\bf h}(a,b)=g(\alpha,\beta)=:\lambda.$ Using \eqref{Eq:coupling0},  we obtain
\begin{eqnarray}\label{Eq:fluctu}
\nonumber
\mathbb{P}( T_{(\alpha n,\beta n)}\le \lambda n+y n^{1/3})&=& \mathbb{P}( \Lr^{\bf h}_{(n(\alpha+h_1\lambda)+h_1 y n^{1/3},n(\beta+h_2\lambda)+h_2 y n^{1/3})}\le \lambda n+y n^{1/3}).\\ 
&=& \mathbb{P}( \Lr^{\bf h}_{(an+h_1y n^{1/3},bn+h_2y n^{1/3})}\le \lambda n+y n^{1/3}).
\end{eqnarray}
Assume now that $a/h_1\le b/h_2$.
Set  
$$N:=n+\frac{h_1}{a}y n^{1/3} \mbox{ so that }n=N -\frac{h_1}{a}yN^{1/3}+o(N^{1/3})$$ 
and define the function $\gamma$ such that 
$$
N\gamma(N)=n+\frac{h_2}{b}y n^{1/3}, 
$$
observe that because of $a/h_1\le b/h_2$ we have
\begin{equation}\label{eq:DefinitionEtMajorationGamma}
N\gamma(N)\leq n+\frac{h_1}{a}y n^{1/3} =N.
\end{equation}
With this notation, \eqref{Eq:fluctu} becomes
$$
\mathbb{P}( \Lr^{{\bf h}}_{(aN,b\gamma(N)N)}\le \lambda N+y(1-\frac{h_1\lambda}{a}) N^{1/3}+o(N^{1/3}))= \mathbb{P}( T_{(\alpha n, \beta n)}\le\lambda n+y n^{1/3}).
$$
Let us notice that 
$$1-\frac{h_1\lambda}{a}=1-\frac{a-\alpha}{a}=\frac{\alpha}{a}>0.$$ Using \eqref{eq:DefinitionEtMajorationGamma} 
we have that $\Lr^{{\bf h}}_{(aN,bN\gamma(N))} \le \Lr^{{\bf h}}_{(aN,bN)}$ and therefore by putting $x=y\alpha/a$ we have for any $x\ge 0$, 
\begin{align*}
\mathbb{P}( \Lr^{{\bf h}}_{(aN,bN)}\le \lambda N+x N^{1/3}+o(N^{1/3}))&\leq \mathbb{P}( \Lr^{{\bf h}}_{(aN,bN\gamma(N))}\le \lambda N+x N^{1/3}+o(N^{1/3}))\\
&\leq  \mathbb{P}( T_{(\alpha n, \beta n)}\le\lambda n+\frac{xa}{\alpha} n^{1/3}).
\end{align*}
Using Johansson's result, we obtain, for $x\ge 0$,
$$\limsup_{N\to \infty}
\mathbb{P}( \mathcal{W}^{{\bf h}}_{(a N,b N)}\le  x )\le \lim_{n\to \infty}\mathbb{P}( T_{(\alpha n, \beta n)}\le \lambda n+x\sigma(\alpha,\beta)\frac{a}{\alpha}  n^{1/3})= F_{TW}(\frac{xa}{\alpha}).
$$
We obtain the lower bound in the same way, setting 
$$\tilde{N}:=n+\frac{h_2}{b}y n^{1/3}$$ 
and  $\tilde{\gamma}$ the function such that 
$$\tilde{N}\tilde{\gamma}(\tilde{N})=n+\frac{h_1}{a}y n^{1/3} .$$
Due to the condition $a/h_1\le b/h_2$, we now have $\tilde{\gamma}(\tilde{N})\ge 1$ for any $y\ge 0$.
The case $x\le 0$ is also obtained with similar arguments.

Finally, note that in the particular case $a/h_1= b/h_2$, we have $\alpha/a=\beta/b$ and we obtain the corollary.
\end{proof}

\noindent {\bf Aknowledgements.} 
The authors are glad to acknowledge N.Georgiou for helpful comments regarding the literature around the problem of fluctuations. 


%
%

\end{document}